\documentclass[letter,11pt,fleqn]{article}
\usepackage{enumitem}
\usepackage{amsmath}
\usepackage{amssymb}
\usepackage{mathtools}
\usepackage{pifont}
\usepackage{theorem} 
\usepackage{euscript}
\usepackage{exscale,relsize}
\usepackage{epic,eepic}
\usepackage{multicol}
\usepackage{tikz}
\usepackage{pgfplots}
\usepackage{makeidx}
\usepackage{srcltx} 
\usepackage{url}
\usepackage{charter}
\usepackage{mathrsfs} 
\usepackage{graphicx}
\usepackage{authblk}
\usepackage{float}
\newcommand{\email}[1]{\href{mailto:#1}{\nolinkurl{#1}}}
\usepackage[normalem]{ulem} 
\newlength{\mySubFigSize}
\definecolor{labelkey}{rgb}{0,0.08,0.45}
\definecolor{refkey}{rgb}{0,0.6,0.0}
\definecolor{Brown}{rgb}{0.45,0.0,0.05}
\definecolor{dgreen}{rgb}{0.00,0.49,0.00}
\definecolor{dblue}{HTML}{0455BF}
\definecolor{orng}{HTML}{D35400}
\definecolor{dred}{HTML}{D90404}
\definecolor{Dblue}{HTML}{8602DC}
\RequirePackage[colorlinks,hyperindex]{hyperref}
\hypersetup{linktocpage=true,citecolor=dblue,linkcolor=dgreen}
\newcommand{\yosi}[2]{\ensuremath{\sideset{^{#2}}{}%
{\operatorname{}}\!\!#1}}
\renewcommand{\leq}{\ensuremath{\leqslant}}
\renewcommand{\geq}{\ensuremath{\geqslant}}

\newcommand{\minimize}[2]{\ensuremath{\underset{\substack{{#1}}}%
{\text{minimize}}\;\;#2 }}
\newcommand{\proxc}[2]{{#1}{\ensuremath{\raisebox{0.2mm}%
{\mbox{\scriptsize\rotatebox[origin=c]{45}%
{\tiny$\,\square\,$}}}{#2}}}}
\newcommand{\proxcc}[2]{{#1}{\ensuremath{\raisebox{0.2mm}%
{\mbox{\scriptsize\rotatebox[origin=c]{45}%
{\tiny$\,\blacksquare\,$}}}{#2}}}}
\newcommand{\proxcg}[2]{{#1}{\ensuremath{
\overset{\hspace{-3mm}\gamma}{{\raisebox{0.2mm}%
{\mbox{\scriptsize\rotatebox[origin=c]{45}%
{\tiny$\,\square\,$}}}{#2}}}}}}
\newcommand{\Scal}[2]{\big\langle{#1}\;\big|\:{#2}\big\rangle} 
\newcommand{\scal}[2]{{\langle{{#1}\mid{#2}}\rangle}}

\newcommand{\abscal}[2]{\left|\left\langle{{#1}\mid{#2}}%
\right\rangle\right|} 
\newcommand{\menge}[2]{\big\{{#1}~\big |~{#2}\big\}}

\newcommand{\Argmin}{\ensuremath{{\text{\rm Argmin}\,}}}
\newcommand{\qq}{\ensuremath{\mathscr{Q}}}
\newcommand{\HH}{\ensuremath{{\mathcal H}}}
\newcommand{\XX}{\ensuremath{{\mathcal X}}}
\newcommand{\YY}{\ensuremath{{\mathcal Y}}}

\newcommand{\GG}{\ensuremath{{\mathcal G}}}
\newcommand{\KK}{\ensuremath{{\mathcal K}}}

\newcommand{\emp}{\ensuremath{{\varnothing}}}
\newcommand{\lev}[1]%
{{\ensuremath{{{{\operatorname{lev}}}_{\leq #1}}\,}}}
\newcommand{\Id}{\ensuremath{\operatorname{Id}}}

\newcommand{\RR}{\ensuremath{\mathbb{R}}}
\newcommand{\RP}{\ensuremath{\left[0,+\infty\right[}}

\newcommand{\BL}{\ensuremath{\EuScript B}\,}
\newcommand{\RPP}{\ensuremath{\left]0,+\infty\right[}}

\newcommand{\RX}{\ensuremath{\left]-\infty,+\infty\right]}}
\newcommand{\RXX}{\ensuremath{\left[-\infty,+\infty\right]}}

\newcommand{\NN}{\ensuremath{\mathbb N}}

\newcommand{\intdom}{\ensuremath{\text{int\,dom}\,}}

\newcommand{\exi}{\ensuremath{\exists\,}}
\newcommand{\moyo}[2]{\ensuremath{\sideset{^{#2}}{}%
{\operatorname{}}\!\!#1}}
\newcommand{\ran}{\ensuremath{\text{\rm ran}\,}}

\newcommand{\cran}{\ensuremath{\overline{\text{\rm ran}}\,}}
\newcommand{\cdom}{\ensuremath{\overline{\text{\rm dom}}\,}}
\newcommand{\zer}{\ensuremath{\text{\rm zer}\,}}
\newcommand{\pinf}{\ensuremath{{+\infty}}}
\newcommand{\minf}{\ensuremath{{-\infty}}}

\newcommand{\dom}{\ensuremath{\text{\rm dom}\,}}

\newcommand{\prox}{\ensuremath{\text{\rm prox}}}

\newcommand{\proj}{\ensuremath{\text{\rm proj}}}

\newcommand{\Fix}{\ensuremath{\text{\rm Fix}\,}}

\newcommand{\gra}{\ensuremath{\text{\rm gra}\,}}
\newcommand{\inte}{\ensuremath{\text{\rm int}\,}}

\newcommand{\infconv}{\ensuremath{\mbox{\small$\,\square\,$}}}
\newcommand{\spushfwd}%
{\ensuremath{\mbox{$\,\triangleright\,$}}}
\newcommand{\pushfwd}%
{\ensuremath{\mbox{\Large$\,\triangleright\,$}}}
\newcommand{\epushfwd}%
{\ensuremath{\mbox{\Large$\,\trianglerightdot\,$}}}

\newcommand{\einfconv}%
{\ensuremath{\mbox{\footnotesize$\,\boxdot\,$}}}
\def\trianglerightdot{{\mkern+1mu \cdot\mkern-6mu \triangleright}}

\def\abstract{\noindent{\bfseries Abstract}. \ignorespaces}

\newtheorem{theorem}{Theorem}[section]
\newtheorem{lemma}[theorem]{Lemma}
\newtheorem{corollary}[theorem]{Corollary}
\newtheorem{proposition}[theorem]{Proposition}

\theoremstyle{plain}{\theorembodyfont{\rmfamily}%
\newtheorem{example}[theorem]{Example}}
\theoremstyle{plain}{\theorembodyfont{\rmfamily}%
\newtheorem{remark}[theorem]{Remark}}
\theoremstyle{plain}{\theorembodyfont{\rmfamily}%
}
\theoremstyle{plain}{\theorembodyfont{\rmfamily}%
}
\theoremstyle{plain}{\theorembodyfont{\rmfamily}%
}
\theoremstyle{plain}{\theorembodyfont{\rmfamily}%
\newtheorem{definition}[theorem]{Definition}}
\theoremstyle{plain}{\theorembodyfont{\rmfamily}%
\newtheorem{problem}[theorem]{Problem}}

\numberwithin{equation}{section}
\begin{document}

\title{\sffamily\huge 
Resolvent and Proximal Compositions\thanks{Contact author: 
P. L. Combettes, {\email{plc@math.ncsu.edu}},
phone:+1 (919) 515 2671. This work was supported by the 
National Science Foundation under grant CCF-2211123.}}

\author{Patrick L.\ Combettes\\ 
\small North Carolina State University,
Department of Mathematics,
Raleigh, NC 27695-8205, USA\\
\small \email{plc@math.ncsu.edu}\\
}

\date{{~}}
\maketitle

\begin{abstract} 
We introduce the resolvent composition, a monotonicity-preserving
operation between a linear operator and a set-valued operator, as
well as the proximal composition, a convexity-preserving operation
between a linear operator and a function. The two operations are
linked by the fact that, under mild assumptions, the
subdifferential of the proximal composition of a convex function is
the resolvent composition of its subdifferential. The resolvent and
proximal compositions are shown to encapsulate known concepts, such
as the resolvent and proximal averages, as well as new operations
pertinent to the analysis of equilibrium problems. A large core of
properties of these compositions is established and several
instantiations are discussed. Applications to the relaxation of
monotone inclusion and convex optimization problems are presented.
\end{abstract}


\newpage
\section{Introduction}
\label{sec:1}

Throughout, $\HH$ and $\GG$ are real Hilbert spaces, $2^{\HH}$ is
the power set of $\HH$, $\Id_\HH$ is the identity operator of
$\HH$, and $\BL(\HH,\GG)$ is the space of bounded linear operators
from $\HH$ to $\GG$. Let $B\colon\GG\to 2^{\GG}$ be a set-valued
operator and denote by $J_B$ its resolvent, that is,
\begin{equation}
\label{e:r}
J_B=(B+\Id_\GG)^{-1}.
\end{equation}
The resolvent operator is a central tool in nonlinear analysis
\cite{Atto84,Livre1,Brez73,Bord18,Roc76a}, largely owing to the
fact that its set of fixed points 
$\menge{y\in\GG}{y\in J_{B}y}$ coincides with the set of
zeros $\menge{y\in\GG}{0\in By}$ of $B$, which models
equilibria in many fields; see for instance
\cite{Barb10,Bri18b,Buim22,Joca22,Svva20,Sign21,%
Facc03,Glow16,Merc80,Roc76b,Zei90B}. 
A standard operation between $B$ and a linear operator
$L\in\BL(\HH,\GG)$ that induces an operator from $\HH$ to $2^{\HH}$
is the composition 
\begin{equation}
\label{e:c}
L^*\circ B\circ L. 
\end{equation}
Early manifestations of this construct can be found in
\cite{Brow69,Rock67}. A somewhat dual operation is the parallel
composition $L^*\pushfwd B\colon\HH\to 2^{\HH}$  defined by
\cite{Livre1,Beck14} (see \cite{Bric23,Bord18} for
further applications)
\begin{equation} 
\label{e:1} 
L^*\pushfwd B=\big(L^*\circ B^{-1}\circ L\big)^{-1}.
\end{equation}
The objective of the present article is to investigate 
alternative compositions, which we call the 
\emph{resolvent composition} and the 
\emph{resolvent cocomposition}.

\begin{definition}
\label{d:1}
Let $L\in\BL(\HH,\GG)$ and $B\colon\GG\to 2^{\GG}$. The
\emph{resolvent composition} of $B$ with $L$ is the operator
$\proxc{L}{B}\colon\HH\to\ 2^{\HH}$ given by 
\begin{equation}
\label{e:2}
\proxc{L}{B}=L^*\pushfwd(B+\Id_{\GG})-\Id_{\HH}
\end{equation}
and the \emph{resolvent cocomposition} of $B$ with $L$ is 
$\proxcc{L}{B}=(\proxc{L}{B^{-1}})^{-1}$.
\end{definition}

The terminology in Definition~\ref{d:1} stems from the following
composition rule, which results from \eqref{e:r}, \eqref{e:2}, and
\eqref{e:1}.

\begin{proposition}
\label{p:0}
Let $L\in\BL(\HH,\GG)$ and $B\colon\GG\to 2^{\GG}$. Then
$J_{\proxc{L}{B}}=L^*\circ J_B\circ L$.
\end{proposition}

The resolvent composition will be shown to encapsulate known
concepts as well as new operations pertinent to the analysis of
equilibrium problems. As an illustration, we recover below the
resolvent average.

\begin{example}[resolvent average]
\label{ex:1}
Let $0\neq p\in\NN$ and, for every $k\in\{1,\ldots,p\}$, let
$B_k\colon\HH\to 2^{\HH}$ and $\omega_k\in\RPP$. Additionally,
let $\GG$ be the standard product vector space $\HH^p$, with
generic element $\boldsymbol{y}=(y_k)_{1\leq k\leq p}$, equipped 
with the scalar product 
$(\boldsymbol{y},\boldsymbol{y}')\mapsto\sum_{k=1}^p
\omega_k\scal{y_k}{y'_k}$, and set 
$L\colon\HH\to\GG\colon x\mapsto (x,\ldots,x)$ and
$B\colon\GG\to 2^\GG\colon\boldsymbol{y}\mapsto
B_1y_1\times\cdots\times B_py_p$. Then 
$L^*\colon\GG\to\HH\colon\boldsymbol{y}\mapsto\sum_{k=1}^p
\omega_ky_k$ and we derive from \eqref{e:2} that
\begin{equation}
\label{e:4}
\proxc{L}{B}=\Bigg(\sum_{k=1}^p\omega_k
\big(B_k+\Id_\HH\big)^{-1}\Bigg)^{-1}-\Id_{\HH}
=\Bigg(\sum_{k=1}^p\omega_kJ_{B_k}\Bigg)^{-1}-\Id_{\HH}.
\end{equation}
In particular, if $\sum_{k=1}^p\omega_k=1$, then \eqref{e:4} is
the \emph{resolvent average} of the operators 
$(B_k)_{1\leq k\leq p}$. This operation is studied in 
\cite{Baus16,Baus13}, while 
$\sum_{k=1}^p\omega_kJ_{B_k}=J_{\proxc{L}{B}}$ shows up in common
zero problems \cite{Cras95,Lehd99}.
\end{example}

Given $L\in\BL(\HH,\GG)$ and a proper convex function
$g\colon\GG\to\RX$ with subdifferential $\partial g$, a question we
shall address is whether the resolvent composition
$\proxc{L}{\partial g}$ is itself a subdifferential operator and,
if so, of which function. Answering this question will lead to the
introduction of the following operations, where $\infconv$ denotes
infimal convolution and where $\qq_\HH=\|\cdot\|_{\HH}^2/2$ and
$\qq_\GG=\|\cdot\|_{\GG}^2/2$ are the canonical quadratic forms of
$\HH$ and $\GG$, respectively (see Section~\ref{sec:2} for
notation).

\begin{definition}
\label{d:2}
Let $L\in\BL(\HH,\GG)$ and let $g\colon\GG\to\RX$ be proper. The
\emph{proximal composition} of $g$ with $L$ is the function
$\proxc{L}{g}\colon\HH\to\RX$ given by 
\begin{equation}
\label{e:12}
\proxc{L}{g}=\big((g^*\infconv\qq_{\GG})\circ L\big)^*-\qq_{\HH}
\end{equation}
and the \emph{proximal cocomposition} of $g$ with $L$ is 
$\proxcc{L}{g}=(\proxc{L}{g^*})^*$.
\end{definition}

In connection with the above question, if $\|L\|\leq 1$ and if $g$
is lower semicontinuous and convex 
in Definition~\ref{d:2}, the proximal
composition will be shown to be linked to the resolvent composition
through the subdifferential identity 
\begin{equation}
\label{e:17}
\partial(\proxc{L}{g})=\proxc{L}{\partial g},
\end{equation}
and its proximity operator to be decomposable as
$\prox_{\proxc{L}{g}}=L^*\circ\prox_g\circ L$, which
explains the terminology in Definition~\ref{d:2}. Furthermore, we
shall see that the proximal composition captures notions such 
as the proximal average of convex functions.

We provide notation and preliminary results in Section~\ref{sec:2}.
Examples of resolvent compositions are presented in
Section~\ref{sec:3}. In Section~\ref{sec:4}, various properties of
the resolvent composition are investigated. Section~\ref{sec:5} is
devoted to the proximal composition and its properties.
Applications to monotone inclusion and variational problems
are discussed in Section~\ref{sec:6}.

\section{Notation and preliminary results}
\label{sec:2}

We refer to \cite{Livre1} for a detailed account of the following
elements of convex and nonlinear analysis. In addition to the
notation introduced in Section~\ref{sec:1}, we designate the direct
Hilbert sum of $\HH$ and $\GG$ by $\HH\oplus\GG$. The scalar
product of a Hilbert space is denoted by $\scal{\cdot}{\cdot}$ and
the associated norm by $\|\cdot\|$. 

Let $A\colon\HH\to 2^{\HH}$ be a set-valued operator. We denote by 
$\gra A=\menge{(x,x^*)\in\HH\times\HH}{x^*\in Ax}$ the graph of 
$A$, by $\dom A=\menge{x\in\HH}{Ax\neq\emp}$ the domain of $A$, by 
$\ran A=\menge{x^*\in\HH}{(\exi x\in\HH)\;x^*\in Ax}$ the range of
$A$, by $\zer A=\menge{x\in\HH}{0\in Ax}$ the set of zeros of $A$,
by $\Fix A=\menge{x\in\HH}{x\in Ax}$ the set of fixed points of
$A$, and by $A^{-1}$ the inverse of $A$, which is the set-valued
operator with graph $\menge{(x^*,x)\in\HH\times\HH}{x^*\in Ax}$. 
The parallel sum of $A$ and $B\colon\HH\to 2^{\HH}$ is
\begin{equation}
\label{e:parasum}
A\infconv B=\big(A^{-1}+B^{-1}\big)^{-1}.
\end{equation}
The resolvent of $A$ is 
$J_A=(A+\Id_\HH)^{-1}=A^{-1}\infconv\Id_\HH$
and the Yosida approximation of $A$ of index $\gamma\in\RPP$ is
$\yosi{A}{\gamma}=\gamma^{-1}(\Id_\HH-J_{\gamma A})$. Furthermore,
$A$ is injective if
\begin{equation}
(\forall x_1\in\HH)(\forall x_2\in\HH)\quad
Ax_1\cap Ax_2\neq\emp\quad\Rightarrow\quad x_1=x_2,
\end{equation}
monotone if 
\begin{equation} 
\label{e:mon2}
\big(\forall (x_1,x_1^*)\in\gra A\big)
\big(\forall (x_2,x_2^*)\in\gra A\big)\quad
\scal{x_1-x_2}{x_1^*-x_2^*}\geq 0,
\end{equation}
$\alpha$-strongly monotone for some $\alpha\in\RPP$ if
$A-\alpha\Id_\HH$ is monotone, and maximally monotone if 
\begin{equation} 
\label{e:maxmon2}
(\forall x_1\in\HH)(\forall x_1^*\in\HH)\quad\big[\:
(x_1,x_1^*)\in\gra A\;\;\Leftrightarrow\;\;
(\forall (x_2,x_2^*)\in\gra A)\;\;
\scal{x_1-x_2}{x_1^*-x_2^*}\geq 0\;\big].
\end{equation}

Let $D$ be a nonempty subset of $\HH$ and let $T\colon D\to\HH$. 
Then $T$ is nonexpansive if it is
$1$-Lipschitzian, firmly nonexpansive if
\begin{equation}
\label{e:f4}
(\forall x_1\in D)(\forall x_2\in D)\quad
\|Tx_1-Tx_2\|^2+\|(\Id_{\HH}-T)x_1-(\Id_{\HH}-T)x_2\|^2
\leq\|x_1-x_2\|^2,
\end{equation}
and strictly nonexpansive if 
\begin{equation}
(\forall x_1\in D)(\forall x_2\in D)\quad
x_1\neq x_2\quad\Rightarrow\quad\|Tx_1-Tx_2\|<\|x_1-x_2\|.
\end{equation}
Let $\beta\in\RPP$. Then $T$ is $\beta$-cocoercive if $\beta T$ is
firmly nonexpansive. 

A function $f\colon\HH\to\RX$ is proper if 
$\dom f=\menge{x\in\HH}{f(x)<\pinf}\neq\emp$, in which case
the set of global minimizers of $f$ is
denoted by $\Argmin f$; if $\Argmin f$ is a singleton, its unique
element is denoted by $\text{argmin}_{x\in\HH}\,f(x)$. 
The conjugate of $f\colon\HH\to\RXX$ is the function 
\begin{equation}
\label{e:conj}
f^*\colon\HH\to\RXX\colon x^*\mapsto\sup_{x\in\HH}
\big(\scal{x}{x^*}-f(x)\big). 
\end{equation}
The infimal convolution of $f\colon\HH\to\RX$ and
$g\colon\HH\to\RX$ is 
\begin{equation}
\label{e:infconv1}
f\infconv g\colon\HH\to\RXX\colon x\mapsto
\inf_{z\in\HH}\big(f(z)+g(x-z)\big)
\end{equation}
and the Moreau envelope of $f$ of index $\gamma\in\RPP$ is
\begin{equation}
\label{e:moyo0}
\moyo{f}{\gamma}=
f\infconv\big(\gamma^{-1}\qq_\HH\big).
\end{equation}
The infimal postcomposition of $f\colon\HH\to\RXX$ by
$L\in\BL(\HH,\GG)$ is
\begin{equation}
\label{e:postcomposition}
L\pushfwd f\colon\GG\to\RXX\colon y\mapsto
\inf f\big(L^{-1}\{y\}\big)=\inf_{\substack{x\in\HH\\ Lx=y}}f(x),
\end{equation}
and it is denoted by $L\epushfwd f$ if, for every $y\in L(\dom f)$,
there exists $x\in\HH$ such that $Lx=y$ and 
$(L\pushfwd f)(y)=f(x)\in\RX$.
We denote by $\Gamma_0(\HH)$ the class of proper
lower semicontinuous convex functions $f\colon\HH\to\RX$. Now let
$f\in\Gamma_0(\HH)$. The subdifferential of $f$ is 
\begin{equation}
\label{e:subdiff}
\partial f\colon\HH\to 2^{\HH}\colon x\mapsto\menge{x^*\in\HH}
{(\forall z\in\HH)\;\;\scal{z-x}{x^*}+f(x)\leq f(z)}
\end{equation}
and its inverse is 
\begin{equation}
\label{e:i}
(\partial f)^{-1}=\partial f^*. 
\end{equation}
Fermat's rule states that
\begin{equation}
\label{e:pierre}
\Argmin f=\zer\partial f.
\end{equation}
The proximity operator of $f$ is 
\begin{equation}
\label{e:dprox}
\prox_{f}=J_{\partial f}\colon\HH\to\HH\colon
x\mapsto\underset{z\in\HH}{\text{argmin}}\;
\bigg(f(z)+\frac{1}{2}\|x-z\|^2\bigg),
\end{equation}
and we have
\begin{equation}
\label{e:fermat}
\Argmin f=\Fix\prox_f.
\end{equation}
We say that $f$ is $\alpha$-strongly convex for some
$\alpha\in\RPP$ if $f-\alpha\qq_\HH$ is convex. 

Let $C$ be a subset of $\HH$. The interior of $C$ is denoted by
$\inte\,C$, the indicator function of $C$ by $\iota_C$, and the
distance function to $C$ by $d_C$. If $C$ is nonempty, closed, and
convex, the projection operator onto $C$ is denoted by $\proj_C$,
i.e., $\proj_C=\prox_{\iota_C}=J_{N_C}$, and the normal cone
operator of $C$ is $N_C=\partial\iota_C$.

Next, we state a few technical facts that will assist us in our
analysis. 

\begin{lemma}
\label{l:9}
Let $L\in\BL(\HH,\GG)$, let $\beta\in\RPP$, let $D$ be a nonempty
subset of $\GG$, and let $T\colon D\to\GG$ be $\beta$-cocoercive.
Then the following hold:
\begin{enumerate}
\item
\label{l:9i}
Suppose that $L\neq 0$. Then
$L^*\circ T\circ L$ is $\beta\|L\|^{-2}$-cocoercive.
\item
\label{l:9ii}
Suppose that $T$ is firmly nonexpansive and that $\|L\|\leq 1$. 
Then $L^*\circ T\circ L$ is firmly nonexpansive.
\item
\label{l:9iii}
Suppose that $D=\HH$, $T$ is firmly nonexpansive, and 
$\|L\|\leq 1$. Then $L^*\circ T\circ L$ is maximally monotone.
\end{enumerate}
\end{lemma}
\begin{proof}
\ref{l:9i}:
Set $R=L^*\circ T\circ L$ and take $x_1$ and $x_2$ in 
$\dom R=L^{-1}(D)$. Then
\begin{align}
\scal{x_1-x_2}{Rx_1-Rx_2}
&=\scal{Lx_1-Lx_2}{T(Lx_1)-T(Lx_2)}\nonumber\\
&\geq\beta\|T(Lx_1)-T(Lx_2)\|^2\nonumber\\
&\geq\beta\|L\|^{-2}\|Rx_1-Rx_2\|^2.
\end{align}

\ref{l:9ii}: The firm nonexpansiveness is clear when $L=0$, and it
otherwise follows from \ref{l:9i} with $\beta=1$.

\ref{l:9iii}: This follows from \ref{l:9ii} and 
\cite[Example~20.30]{Livre1}.
\end{proof}

The following result, essentially due to Minty \cite{Mint62}, 
illuminates the interplay between nonexpansiveness and
monotonicity.

\begin{lemma}{\rm(\cite[Proposition~23.8]{Livre1})}
\label{l:0}
Let $D$ be a nonempty subset of $\HH$, let $T\colon D\to\HH$,
and set $A=T^{-1}-\Id_\HH$. Then the following hold:
\begin{enumerate}
\item
\label{l:0i}
$T=J_A$. 
\item 
\label{l:0ii}
$T$ is firmly nonexpansive if and only if $A$ is monotone.
\item 
\label{l:0iii}
$T$ is firmly nonexpansive and $D=\HH$ if and only if $A$ 
is maximally monotone. 
\end{enumerate}
\end{lemma}

\begin{lemma}
\label{l:1}
Let $A\colon\HH\to 2^{\HH}$. Then the following hold:
\begin{enumerate}
\item
\label{l:1i-}
Let $\gamma\in\RPP$. Then 
$\yosi{A}{\gamma}=(\gamma\Id_\HH+A^{-1})^{-1}
=(J_{\gamma^{-1}A^{-1}})\circ\gamma^{-1}\Id_\HH$.
\item
\label{l:1i}
$\Id_\HH\infconv A=J_{A^{-1}}=\Id_\HH-J_A$.
\item
\label{l:1ii}
$\zer A=\Fix J_A$.
\item
\label{l:1iii}
$(A-\Id_{\HH})^{-1}=(\Id_\HH-A^{-1})^{-1}-\Id_\HH$.
\item
\label{l:1iv}
Suppose that $A$ is monotone and let $\alpha\in\RPP$. 
Then $A$ is $\alpha$-strongly monotone if and only if
$J_A$ is $(\alpha+1)$-cocoercive. 
\end{enumerate}
\end{lemma}
\begin{proof}
\ref{l:1i-}: See \cite[Proposition~23.7(ii)]{Livre1}.

\ref{l:1i}: Apply \ref{l:1i-} with $\gamma=1$.

\ref{l:1ii}:
$\zer A=\menge{x\in\HH}{x\in (A+\Id_\HH)x}=
\menge{x\in\HH}{x\in(A+\Id_\HH)^{-1}x}$.

\ref{l:1iii}: By \ref{l:1i},
$A^{-1}=J_{A-\Id_\HH}=\Id_\HH-J_{(A-\Id_\HH)^{-1}}
=\Id_\HH-(\Id_\HH+(A-\Id_\HH)^{-1})^{-1}$.
So $(\Id_\HH-A^{-1})^{-1}=\Id_\HH+(A-\Id_\HH)^{-1}$ as claimed.

\ref{l:1iv}: See \cite[Proposition~23.13]{Livre1}.
\end{proof}

\begin{lemma}{\rm(\cite[Theorem~2.1]{Alim14})}
\label{l:2}
Let $A\colon\HH\to 2^{\HH}$ be a maximally monotone operator and
let $B\colon\HH\to 2^{\HH}$ be a monotone operator such that 
$\dom B=\HH$ and $A-B$ is monotone. Then $A-B$ is maximally
monotone.
\end{lemma}

\begin{lemma}{\rm(\cite[Theorem~5]{Penn01})}
\label{l:3}
Let $L\in\BL(\HH,\GG)$ and let $B\colon\GG\to 2^{\GG}$ be
$3^*$~monotone, that is,
\begin{equation}
\label{e:3*mon}
\big(\forall (y_1,y_1^*)\in\dom B\times\ran B\big)
\;\;\text{\rm sup}
\menge{\scal{y_1-y_2}{y_2^*-y_1^*}}{(y_2,y_2^*)\in\gra B}<\pinf.
\end{equation}
Suppose that $L^*\circ B\circ L$ is
maximally monotone. Then the following hold:
\begin{enumerate}
\item
\label{l:3i}
$\inte L^*(\ran B)\subset\ran(L^*\circ B\circ L)$.
\item
\label{l:3ii}
$L^*(\ran B)\subset\cran(L^*\circ B\circ L)$.
\end{enumerate}
\end{lemma}

\begin{lemma}
\label{l:7}
Let $f\in\Gamma_0(\HH)$. Then the following hold:
\begin{enumerate}
\item
\label{l:7i}
{\rm\cite[Theorem~9.20]{Livre1}}
$f$ admits a continuous affine minorant.
\item
\label{l:7ii}
{\rm\cite[Corollary~13.38]{Livre1}}
$f^*\in\Gamma_0(\HH)$ and $f^{**}=f$.
\end{enumerate}
\end{lemma}

\begin{lemma}{\rm\cite[Theorem~18.15]{Livre1}}
\label{l:8}
Let $f\colon\HH\to\RR$ be continuous and convex, and 
let $\beta\in\RPP$. Then the following are equivalent:
\begin{enumerate}
\item
\label{l:8i}
$f$ is Fr\'echet differentiable on $\HH$ and $\nabla f$ is
$\beta$-Lipschitz continuous.
\item
\label{l:8ii}
$f^*$ is $\beta^{-1}$-strongly convex.
\end{enumerate}
\end{lemma}

\begin{lemma}{\rm(Moreau \cite{More65})}
\label{l:5}
Let $T\colon\HH\to\HH$ be nonexpansive. Then $T$ is a proximity
operator if and only if there exists a differentiable convex
function $h\colon\HH\to\RR$ such that $T=\nabla h$. In this case,
$T=\prox_{f}$, where $f=h^*-\qq_\HH$. 
\end{lemma}

\begin{lemma}{\rm(Moreau \cite{More65})}
\label{l:4}
Let $f\in\Gamma_0(\HH)$. Then the following hold:
\begin{enumerate}
\item
\label{l:4i--}
$\partial f$ is maximally monotone.
\item
\label{l:4i-}
$f\infconv\qq_\HH\colon\HH\to\RR$ is convex and Fr\'echet
differentiable.
\item
\label{l:4i+}
$(f\infconv\qq_\HH)^*=f^*+\qq_\HH$ and 
$(f+\qq_\HH)^*=f^*\infconv\qq_\HH$.
\item
\label{l:4i}
$\prox_f=\nabla(f^*\infconv\qq_\HH)$.
\item
\label{l:4i++}
$\prox_f$ is firmly nonexpansive. 
\item
\label{l:4ii}
$f\infconv\qq_\HH+f^*\infconv\qq_\HH=\qq_\HH$.
\item
\label{l:4iii}
$\prox_f+\prox_{f^*}=\Id_\HH$.
\item
\label{l:4vi}
$\partial(f+\qq_{\HH})=\partial f+\Id_{\HH}$.
\end{enumerate}
\end{lemma}

\section{Examples of resolvent compositions}
\label{sec:3}

We provide a few examples that expose various facets of the
resolvent composition. The first one describes a scenario in which
the compositions \eqref{e:c}, \eqref{e:1}, and \eqref{e:2} happen
to coincide.

\begin{example}
\label{ex:6}
Suppose that $L\in\BL(\HH,\GG)$ is a surjective isometry and 
let $B\colon\GG\to 2^{\GG}$. Then 
$\proxc{L}{B}=L^*\circ B\circ L=L^*\pushfwd B$.  
\end{example}
\begin{proof}
Since $L^{-1}=L^*$, \eqref{e:2} yields 
$\proxc{L}{B}
=L^*\pushfwd(B+\Id_{\GG})-\Id_{\HH}
=(L^{-1}\circ(B+\Id_{\GG})^{-1}\circ L)^{-1}-\Id_{\HH}
=L^{-1}\circ(B+\Id_{\GG})\circ L-\Id_{\HH}
=L^{-1}\circ B\circ L
=(L^{-1}\circ B^{-1}\circ L)^{-1}
=L^{-1}\pushfwd B
=L^*\pushfwd B$.
\end{proof}

\begin{example}
\label{ex:0}
Let $\alpha\in\RR\smallsetminus\{0\}$, let 
$B\colon\HH\to 2^{\HH}$, and set
$L=\alpha^{-1}\Id_\HH$. Then 
$\proxc{L}{B}=(\alpha^2-1)\Id_\HH+\alpha B\circ(\alpha\Id_\HH)$.
\end{example}

The broad potential of Definition~\ref{d:1} is illustrated below by
deploying it in product spaces.

\begin{example}[multivariate resolvent mixture]
\label{ex:90}
Let $0\neq m\in\NN$ and $0\neq p\in\NN$. For every
$i\in\{1,\ldots,m\}$ and every $k\in\{1,\ldots,p\}$, let $\HH_i$
and $\GG_k$ be real Hilbert spaces, let
$L_{ki}\in\BL(\HH_i,\GG_k)$, let $\omega_k\in\RPP$, and let
$B_k\colon\GG_k\to 2^{\GG_k}$. Let $\HH$ be the standard 
product vector space $\HH_1\times\cdots\times\HH_m$, with generic
element $\boldsymbol{x}=(x_i)_{1\leq i\leq m}$, and equipped with
the scalar product 
$(\boldsymbol{x},\boldsymbol{x}')\mapsto\sum_{i=1}^m
\scal{x_i}{x'_i}$. Let $\GG$ be the standard product vector space
$\GG_1\times\cdots\times\GG_p$, with generic element
$\boldsymbol{y}=(y_k)_{1\leq k\leq p}$, and equipped with the
scalar product $(\boldsymbol{y},\boldsymbol{y}')\mapsto
\sum_{k=1}^p\omega_k\scal{y_k}{y'_k}$. Set
\begin{equation}
\label{e:60}
L\colon\HH\to\GG\colon\boldsymbol{x}\mapsto\Bigg(
\sum_{i=1}^mL_{1i}x_i,\ldots,\sum_{i=1}^mL_{pi}x_i\Bigg)
\end{equation}
and
\begin{equation}
\label{e:61}
B\colon\GG\to 2^{\GG}\colon\boldsymbol{y}\mapsto
B_1y_1\times\cdots\times B_py_p.
\end{equation}
Then Proposition~\ref{p:0} yields\\
\begin{equation}
\label{e:74}
J_{\proxc{L}{B}}\colon\HH\to 2^{\HH}\colon\boldsymbol{x}\mapsto
\Bigg(\sum_{k=1}^p\omega_kL_{k1}^*
\bigg(J_{B_k}\bigg(\sum_{i=1}^mL_{ki}x_i\bigg)\bigg),\ldots,
\sum_{k=1}^p\omega_kL_{km}^*
\bigg(J_{B_k}\bigg(\sum_{i=1}^mL_{ki}x_i\bigg)\bigg)\Bigg)
\end{equation}
and we call $\proxc{L}{B}=(J_{\proxc{L}{B}})^{-1}-\Id_\HH$ a
\emph{multivariate resolvent mixture}.
\end{example}

When $m=1$ in Example~\ref{ex:90}, we obtain the following
construction.

\begin{example}[resolvent mixture]
\label{ex:5}
Let $0\neq p\in\NN$ and, for every 
$k\in\{1,\ldots,p\}$, let $\GG_k$ be a real Hilbert
space, let $L_{k}\in\BL(\HH,\GG_k)$, let $\omega_k\in\RPP$, and
let $B_k\colon\GG_k\to 2^{\GG_k}$. Define $\GG$ and $B$ as in
Example~\ref{ex:90}, and set
$L\colon\HH\to\GG\colon x\mapsto(L_1x,\ldots,L_px)$.
Then we obtain the \emph{resolvent mixture}
\begin{equation}
\label{e:62}
\proxc{L}{B}
=\Bigg(\sum_{k=1}^p\omega_kL_k^*\circ J_{B_k}\circ L_k\Bigg)^{-1}
-\Id_\HH
=\Bigg(\sum_{k=1}^p\omega_kJ_{\proxc{L_k}{B_k}}\Bigg)^{-1}
-\Id_\HH
\end{equation}
and $J_{\proxc{L}{B}}=\sum_{k=1}^p\omega_kL_{k}^*\circ
J_{B_k}\circ L_k$.
In particular, if, for every $k\in\{1,\ldots,p\}$, $\GG_k=\HH$ and
$L_k=\Id_{\HH}$, then \eqref{e:62} reduces to Example~\ref{ex:1},
which itself encompasses the resolvent average.
\end{example}

\begin{example}[linear projector]
\label{ex:9}
Let $V$ be a closed vector subspace of $\HH$ and let  
$B\colon\HH\to 2^{\HH}$. Then $\proxc{\proj_V}{B}=
(\proj_V\circ J_B\circ\proj_V)^{-1}-\Id_\HH$. Here are noteworthy
special cases of this construction:
\begin{enumerate}
\item
\label{ex:9i}
Let $C$ be a nonempty closed convex subset of $\HH$ and suppose
that $B=N_C$. Then $\proxc{\proj_V}{B}=
(\proj_V\circ\proj_C\circ\proj_V)^{-1}-\Id_\HH$. This operator was
employed in \cite{Sico20} to construct an instance of weak -- 
but not strong -- convergence of the Douglas-Rachford algorithm.
\item
\label{ex:9ii}
Define $\GG$, $(B_k)_{1\leq k\leq p}$, and $B$ as in 
Example~\ref{ex:1}, with $\sum_{k=1}^p\omega_k=1$. In addition,
define $V=\menge{\boldsymbol{y}\in\GG}{y_1=\cdots=y_p}$, let
$A=(\sum_{k=1}^p\omega_kJ_{B_k})^{-1}-\Id_\HH$ be the resolvent
average of $(B_k)_{1\leq k\leq p}$ (see \eqref{e:4}),
let $\boldsymbol{y}\in\GG$, and set 
$\overline{y}=\sum_{k=1}^p\omega_ky_k$. Then we derive from
Proposition~\ref{p:0} and 
\cite[Propositions~23.18 and 29.16]{Livre1} that
$J_{\proxc{\proj_V}{B}}\,\boldsymbol{y}
=(\sum_{k=1}^p\omega_kJ_{B_k}\overline{y},\ldots,
\sum_{k=1}^p\omega_kJ_{B_k}\overline{y})
=(J_A\overline{y},\ldots,J_A\overline{y})$.
In the case of convex feasibility problems, where each $B_k$ is the
normal cone to a nonempty closed convex set, this type of
construction was first proposed in \cite{Pier76,Pier84}.
\end{enumerate}
\end{example}

The next example places the subdifferential identity \eqref{e:17}
in a rigorous framework.

\begin{example}[subdifferential]
\label{ex:8}
Suppose that $L\in\BL(\HH,\GG)$ satisfies $0<\|L\|\leq 1$ and
let $g\colon\GG\to\RX$ be a proper function that admits a 
continuous affine minorant. Then the following hold:
\begin{enumerate}
\item
\label{ex:8i-}
$g^*\in\Gamma_0(\GG)$.
\item
\label{ex:8i}
$\proxc{L}{g}\in\Gamma_0(\HH)$.
\item
\label{ex:8ii-}
$\proxc{L}{\partial g^{**}}=\partial(\proxc{L}{g})$.
\item
\label{ex:8iii-}
$\prox_{\proxc{L}{g}}=L^*\circ\prox_{g^{**}}\circ L$.
\item
\label{ex:8ii}
Suppose that $g\in\Gamma_0(\GG)$. Then
$\proxc{L}{\partial g}=\partial(\proxc{L}{g})$.
\item
\label{ex:8iii}
Suppose that $g\in\Gamma_0(\GG)$. Then
$\prox_{\proxc{L}{g}}=L^*\circ\prox_g\circ L$.
\end{enumerate}
\end{example}
\begin{proof}
Set $h=((g^*\infconv\qq_{\GG})\circ L)^*-\|L\|^{-2}\qq_{\HH}$. On
the one hand, by \cite[Proposition~13.13]{Livre1}, $g^*$ is 
lower semicontinuous and convex. On the other hand, by
\cite[Propositions~13.10(ii) and 13.12(ii)]{Livre1}, $g^*$ is
proper. Thus, 
\begin{equation}
\label{e:g2}
g^*\in\Gamma_0(\GG)
\end{equation}
and it follows from 
Lemma~\ref{l:4} that $g^*\infconv\qq_{\GG}\colon\GG\to\RR$ is
Fr\'echet differentiable on $\GG$ with nonexpansive gradient
$\Id_{\GG}-\prox_{g^*}$. In turn, 
\begin{equation}
\nabla\big((g^*\infconv\qq_{\GG})\circ L\big)=
L^*\circ(\Id_{\GG}-\prox_{g^*})\circ L
\end{equation}
has Lipschitz constant $\|L\|^2$ and we derive from Lemma~\ref{l:8}
that 
\begin{equation}
\label{e:a3}
\big((g^*\infconv\qq_{\GG})\circ L\big)^*\;\;\text{is}\;\;
\|L\|^{-2}\text{-strongly convex}, 
\end{equation}
We also record the fact that \eqref{e:g2} and
Lemma~\ref{l:7}\ref{l:7ii} imply that $g^{**}\in\Gamma_0(\GG)$.

\ref{ex:8i-}: See \eqref{e:g2}.

\ref{ex:8i}: We infer from \eqref{e:a3} that 
$h\in\Gamma_0(\HH)$. Hence, since $\|L\|^{-2}>1$, we conclude that 
\begin{equation}
\proxc{L}{g}=h+\big(\|L\|^{-2}-1\big)\qq_{\HH}\in\Gamma_0(\HH).
\end{equation}

\ref{ex:8ii-}:
Note that, on account of Lemma~\ref{l:7}\ref{l:7i}, $g^{**}$
admits a continuous affine minorant. Using 
\eqref{e:2}, \eqref{e:1}, \eqref{e:i},
Lemma~\ref{l:4}\ref{l:4vi}, Lemma~\ref{l:4}\ref{l:4i+},
\cite[Proposition~13.16(iii) and Corollary~16.53(i)]{Livre1}, 
we get 
\begin{align}
\label{e:a4}
\proxc{L}{\partial g^{**}}+\Id_\HH
&=L^*\pushfwd(\partial g^{**}+\Id_{\GG})\nonumber\\
&=\big(L^*\circ\big(\partial g^{**}+\Id_{\GG}\big)^{-1}
\circ L\big)^{-1}\nonumber\\
&=\big(L^*\circ\big(\partial(g^{**}+\qq_{\GG})\big)^{-1}
\circ L\big)^{-1}\nonumber\\
&=\big(L^*\circ\partial(g^{**}+\qq_{\GG})^{*}
\circ L\big)^{-1}\nonumber\\
&=\big(L^*\circ\partial(g^{***}\infconv\qq_{\GG})
\circ L\big)^{-1}\nonumber\\
&=\big(L^*\circ\partial(g^{*}\infconv\qq_{\GG})
\circ L\big)^{-1}\nonumber\\
&=\partial\big((g^*\infconv\qq_{\GG})
\circ L\big)^{*}.
\end{align}
Since $0<\|L\|\leq 1$, we deduce from \eqref{e:a3} that 
$((g^*\infconv\qq_{\GG})\circ L)^*-\qq_{\HH}\in\Gamma_0(\HH)$. 
Hence, appealing to Lemma~\ref{l:4}\ref{l:4vi} and 
\eqref{e:12}, we obtain
\begin{align}
\label{e:a5}
\partial\big((g^*\infconv\qq_{\GG})\circ L\big)^*
&=\partial\Big(\big((g^*\infconv\qq_{\GG})\circ L\big)^*-
\qq_{\HH}+\qq_\HH\Big)
\nonumber\\
&=\partial\Big(\big((g^*\infconv\qq_{\GG})\circ L)^*
-\qq_{\HH}\Big)+\partial\qq_\HH\nonumber\\
&=\partial(\proxc{L}{g})+\Id_\HH.
\end{align}
The sought identity follows by combining \eqref{e:a4} 
and \eqref{e:a5}.

\ref{ex:8iii-}: 
In view of \ref{ex:8i}, $\prox_{\proxc{L}{g}}$ is
well defined and, combining \ref{ex:8ii-} and
Proposition~\ref{p:0}, we obtain
$\prox_{\proxc{L}{g}}
=J_{\partial(\proxc{L}{g})}
=J_{\proxc{L}{\partial g^{**}}}
=L^*\circ J_{\partial g^{**}}\circ L
=L^*\circ\prox_{g^{**}}\circ L$.

\ref{ex:8ii}--\ref{ex:8iii}: These identities follow from
Lemma~\ref{l:7}\ref{l:7ii}, \ref{ex:8ii-}, and \ref{ex:8iii-}.
\end{proof}

\begin{example}[proximity operator]
\label{ex:12}
Suppose that $L\in\BL(\HH,\GG)$ satisfies $0<\|L\|\leq 1$ and
let $g\in\Gamma_0(\GG)$. Then we derive from 
Lemma~\ref{l:4}\ref{l:4i} and Example~\ref{ex:8}\ref{ex:8ii} that
\begin{equation}
\proxc{L}{\prox_g}=
\partial\big(\proxc{L}{(g^*\infconv\qq_{\GG})}\big).
\end{equation}
\end{example}

\begin{example}[projection operator]
\label{ex:13}
Suppose that $L\in\BL(\HH,\GG)$ satisfies $0<\|L\|\leq 1$, let 
$C$ be a nonempty closed convex subset of $\GG$, and set 
$g=\iota_C$. Then $g\infconv\qq_\GG=d_C^2/2$ and 
Lemma~\ref{l:4}\ref{l:4ii} yields
$g^*\infconv\qq_\GG=\qq_\GG-d_C^2/2$. Altogether, we derive
from Example~\ref{ex:12} and Lemma~\ref{l:4}\ref{l:4iii} that
\begin{equation}
\proxc{L}{\proj_C}=
\partial\big(\proxc{L}{(\qq_\GG-d_C^2/2)}\big)
\quad\text{and}\quad
\proxc{L}{(\Id_{\GG}-\proj_C)}=
\partial\big(\proxc{L}{(d_C^2/2)}\big).
\end{equation}
\end{example}

\begin{example}[frames]
\label{ex:11}
Suppose that $(e_k)_{k\in\NN}$ is a frame in $\HH$ \cite{Chri08},
i.e., there exist $\alpha\in\RPP$ and $\beta\in\RPP$ such that
\begin{equation}
\label{e:f}
(\forall x\in\HH)\quad 
\alpha\|x\|^2\leq\sum_{k\in\NN}\abscal{x}{e_k}^2\leq\beta\|x\|^2. 
\end{equation}
We set $\GG=\ell^2(\NN)$, denote by
$L\colon\HH\to\GG\colon x\mapsto
(\scal{x}{e_k})_{k\in\NN}$ the frame analysis operator, and
let $(\phi_k)_{k\in\NN}$ be functions in
$\Gamma_0(\RR)$ such that $(\forall k\in\NN)$
$\phi_k\geq\phi_k(0)=0$. Further, we set
$B\colon\GG\to 2^{\GG}\colon (\eta_k)_{k\in\NN}\mapsto
\menge{(\nu_k)_{k\in\NN}\in\GG}{(\forall k\in\NN)\;
\nu_k\in\partial\phi_k(\eta_k)}$. Then
\begin{equation}
\proxc{L}{B}=\Bigg(\sum_{k\in\NN}\big(\prox_{\phi_k}
\scal{\cdot}{e_k}\big)e_k\Bigg)^{-1}-\Id_\HH.
\end{equation}
\end{example}
\begin{proof}
Set $\varphi\colon\GG\to\RX\colon(\eta_k)_{k\in\NN}\mapsto
\sum_{k\in\NN}\phi_k(\eta_k)$ and note that
$L^*\colon\GG\to\HH\colon(\eta_k)_{k\in\NN}
\mapsto\sum_{k\in\NN}\eta_ke_k$. As shown in \cite{Smms05}, 
$\varphi\in\Gamma_0(\GG)$, $B=\partial\varphi$, and 
$J_B\colon(\eta_k)_{k\in\NN}\mapsto
(\prox_{\phi_k}\eta_k)_{k\in\NN}$. Thus,
$(L^*\pushfwd(B+\Id_{\GG}))^{-1}=
L^*\circ J_B\circ L=\sum_{k\in\NN}\big(\prox_{\phi_k}
\scal{\cdot}{e_k}\big)e_k$.
\end{proof}

Our last example parallels Example~\ref{ex:8} in the case of
the proximal cocomposition of Definition~\ref{d:2}.

\begin{example}[subdifferential]
\label{ex:7}
Suppose that $L\in\BL(\HH,\GG)$ satisfies $0<\|L\|\leq 1$ and
let $g\colon\GG\to\RX$ be a proper function that admits a 
continuous affine minorant. Then the following hold:
\begin{enumerate}
\item
\label{ex:7i}
$\proxcc{L}{g}\in\Gamma_0(\HH)$.
\item
\label{ex:7ii-}
$\proxcc{L}{\partial g^{**}}=\partial(\proxcc{L}{g})$.
\item
\label{ex:7iii-}
$\prox_{\proxcc{L}{g}}=
\Id_{\HH}-L^*\circ L+L^*\circ\prox_{g^{**}}\circ L$.
\item
\label{ex:7ii}
Suppose that $g\in\Gamma_0(\GG)$. Then
$\proxcc{L}{\partial g}=\partial(\proxcc{L}{g})$.
\item
\label{ex:7iii}
Suppose that $g\in\Gamma_0(\GG)$. Then
$\prox_{\proxcc{L}{g}}=
\Id_{\HH}-L^*\circ L+L^*\circ\prox_{g}\circ L$.
\end{enumerate}
\end{example}
\begin{proof}
By virtue of Example~\ref{ex:8}\ref{ex:8i-} and 
Lemma~\ref{l:7}\ref{l:7i}, $g^*$ is in $\Gamma_0(\GG)$ and it
admits a continuous affine minorant. As a consequence of 
Example~\ref{ex:8}\ref{ex:8i}, we record the fact that 
\begin{equation}
\label{e:A}
\proxc{L}{g^*}\in\Gamma_0(\HH). 
\end{equation}

\ref{ex:7i}: We invoke \eqref{e:A} and Lemma~\ref{l:7}\ref{l:7ii}
to deduce that $\proxcc{L}{g}=(\proxc{L}{g^*})^*\in\Gamma_0(\HH)$. 

\ref{ex:7ii-}: It follows from Lemma~\ref{l:7}\ref{l:7ii} that
$g^{**}\in\Gamma_0(\GG)$. Hence, using Definition~\ref{d:1},
\eqref{e:A}, \eqref{e:i}, and Definition~\ref{d:2}, we obtain
\begin{equation}
\proxcc{L}{\partial g^{**}}=
\big(\proxc{L}{(\partial g^{**})^{-1}}\big)^{-1}=
\big(\proxc{L}{\partial g^{***}}\big)^{-1}=
\big(\partial(\proxc{L}{g^*})\big)^{-1}=
\partial(\proxc{L}{g^*})^{*}=
\partial(\proxcc{L}{g}).
\end{equation}

\ref{ex:7iii-}: 
Property~\ref{ex:7i} ensures that $\prox_{\proxcc{L}{g}}$ is well
defined. Further, we deduce from \eqref{e:A}, 
Lemma~\ref{l:4}\ref{l:4iii}, and
Example~\ref{ex:8}\ref{ex:8iii-} that
\begin{equation}
\prox_{\proxcc{L}{g}}
=\Id_\HH-\prox_{\proxc{L}{g^*}}
=\Id_\HH-L^*\circ\prox_{g^{***}}\circ L
=\Id_\HH-L^*\circ(\Id_\GG-\prox_{g^{**}})\circ L.
\end{equation}

\ref{ex:7ii}--\ref{ex:7iii}: Since $g=g^{**}$ by 
Lemma~\ref{l:7}\ref{l:7ii}, these follow from \ref{ex:7ii-} and
\ref{ex:7iii-}.
\end{proof}

\section{Properties of the resolvent composition}
\label{sec:4}

We start with basic facts. 

\begin{proposition}
\label{p:1}
Let $L\in\BL(\HH,\GG)$ and let $B\colon\GG\to 2^{\GG}$. Then the
following hold:
\begin{enumerate}
\item
\label{p:1i}
$\proxc{L}{B}=(L^*\circ J_B\circ L)^{-1}-\Id_{\HH}$.
\item
\label{p:1ic}
$\proxcc{L}{B}=(\Id_\HH-L^*\circ L+
L^*\circ J_B\circ L)^{-1}-\Id_{\HH}$.
\item
\label{p:1iic}
Suppose that $L$ is an isometry. Then 
$\proxc{L}{B}=\proxcc{L}{B}$.
\item
\label{p:1ii}
$(\proxc{L}{B})^{-1}=\proxcc{L}{B}^{-1}
=(\Id_{\HH}-L^*\circ J_B\circ L)^{-1}-\Id_\HH$.
\item
\label{p:1ic2}
$J_{\proxcc{L}{B}}=\Id_\HH-L^*\circ L+L^*\circ J_B\circ L$.
\item
\label{p:1i+}
$\gra(\proxc{L}{B})=\menge{(x,x^*)\in\HH\times\HH}
{(x+x^*,x)\in\gra(L^*\circ J_B\circ L)}$.
\item
\label{p:1i+c}
$\gra(\proxcc{L}{B})=\menge{(x,x^*)\in\HH\times\HH}
{(x+x^*,(L^*\circ L)(x+x^*)-x^*)\in\gra(L^*\circ J_B\circ L)}$.
\item
\label{p:1iii}
$\dom(\proxc{L}{B})\subset L^*(\dom B)$. 
\item
\label{p:1iv}
$\ran(\proxc{L}{B})\subset\ran(\Id_\HH-L^*\circ L)+L^*(\ran B)$. 
\item
\label{p:1iiic}
$\dom(\proxcc{L}{B})\subset\ran(\Id_\HH-L^*\circ L)+ L^*(\dom B)$. 
\item
\label{p:1ivc}
$\ran(\proxcc{L}{B})\subset L^*(\ran B)$. 
\item
\label{p:1v}
$\zer(\proxc{L}{B})=\Fix(L^*\circ J_B\circ L)$.
\item
\label{p:1vii}
$L^{-1}(\zer B)\subset\zer(\proxcc{L}{B})$.
\item
\label{p:1vi}
$(\proxc{L}{B})\infconv\Id_{\HH}+
L^*\circ(B^{-1}\infconv\Id_{\GG})\circ L=\Id_{\HH}$.
\item
\label{p:1x}
$(\proxcc{L}{B})\infconv\Id_{\HH}
=L^*\circ(B\infconv\Id_{\GG})\circ L$.
\end{enumerate}
\end{proposition}
\begin{proof}
\ref{p:1i}: A consequence of \eqref{e:r} and
Proposition~\ref{p:0}.

\ref{p:1ic}: In view of \ref{p:1i}, Lemma~\ref{l:1}\ref{l:1iii},
 and Lemma~\ref{l:1}\ref{l:1i},
$\proxcc{L}{B}=(\proxc{L}{B^{-1}})^{-1}
=((L^*\circ J_{B^{-1}}\circ L)^{-1}-\Id_{\HH})^{-1}
=(\Id_{\HH}-L^*\circ J_{B^{-1}}\circ L)^{-1}-\Id_\HH
=(\Id_{\HH}-L^*\circ(\Id_\GG-J_B)\circ L)^{-1}-\Id_\HH$.

\ref{p:1iic}: Since $L^*\circ L=\Id_\HH$, this follows from
\ref{p:1i} and \ref{p:1ic}.

\ref{p:1ii}: The first identity is clear by inspecting
Definition~\ref{d:1}. To establish the second, note that \ref{p:1i}
and Lemma~\ref{l:1}\ref{l:1iii} yield
\begin{equation}
(\proxc{L}{B})^{-1}
=\big((L^*\circ J_B\circ L)^{-1}-\Id_\HH\big)^{-1}
=(\Id_{\HH}-L^*\circ J_B\circ L)^{-1}-\Id_\HH.
\end{equation}

\ref{p:1ic2}: A consequence of \ref{p:1ic}.

\ref{p:1i+}: 
Let $(x,x^*)\in\HH\times\HH$. Then \ref{p:1i} yields
$(x,x^*)\in\gra(\proxc{L}{B})$ $\Leftrightarrow$
$x^*\in(L^*\circ J_B\circ L)^{-1}x-x$
$\Leftrightarrow$ $x\in(L^*\circ J_B\circ L)(x+x^*)$.

\ref{p:1i+c}: 
Let $(x,x^*)\in\HH\times\HH$. By \ref{p:1i+} and 
Lemma~\ref{l:1}\ref{l:1i}, 
$(x,x^*)\in\gra(\proxcc{L}{B})$ $\Leftrightarrow$
$(x^*,x)\in\gra(\proxc{L}{B^{-1}})$ $\Leftrightarrow$
$x+x^*\in(L^*\circ J_{B^{-1}}\circ L)^{-1}x^*$ $\Leftrightarrow$
$x^*\in(L^*\circ J_{B^{-1}}\circ L)(x+x^*)
=(L^*\circ L)(x+x^*)-(L^*\circ J_B\circ L)(x+x^*)$ 
$\Leftrightarrow$
$(L^*\circ L)(x+x^*)-x^*\in(L^*\circ J_B\circ L)(x+x^*)$.

\ref{p:1iii}: 
In view of \ref{p:1i} and Proposition~\ref{p:0}, 
\begin{equation}
\label{e:h4}
\dom(\proxc{L}{B})
=\dom(L^*\circ J_B\circ L)^{-1}
=\ran(L^*\circ J_B\circ L)\subset L^*(\ran J_B)
=L^*(\dom B).
\end{equation}

\ref{p:1iv}:
We invoke \ref{p:1ii} and Lemma~\ref{l:1}\ref{l:1i} to get
\begin{align}
\ran(\proxc{L}{B})
&=\dom(\proxc{L}{B})^{-1}
\nonumber\\
&=\dom\big(\Id_{\HH}-L^*\circ J_B\circ L\big)^{-1}
\nonumber\\
&=\ran\big(\Id_\HH-L^*\circ J_B\circ L\big)
\nonumber\\
&=\ran\big(\Id_\HH-L^*\circ L+L^*\circ(\Id_\GG-J_B)\circ L\big)
\nonumber\\
&=\ran\big(\Id_\HH-L^*\circ L+L^*\circ J_{B^{-1}}\circ L\big)
\label{e:h24}\\
&\subset\ran(\Id_\HH-L^*\circ L)+\ran(L^*\circ J_{B^{-1}}\circ L)
\nonumber\\
&\subset\ran(\Id_\HH-L^*\circ L)+L^*(\ran J_{B^{-1}})
\nonumber\\
&=\ran(\Id_\HH-L^*\circ L)+L^*(\dom {B^{-1}})
\nonumber\\
&=\ran(\Id_\HH-L^*\circ L)+L^*(\ran B),
\label{e:h5}
\end{align}
which furnishes the desired inclusion.

\ref{p:1iiic}:
In view of \ref{p:1iv},
$\dom(\proxcc{L}{B})=\ran(\proxc{L}{B^{-1}})
\subset\ran(\Id_\HH-L^*\circ L)+L^*(\ran B^{-1})=
\ran(\Id_\HH-L^*\circ L)+L^*(\dom B)$.

\ref{p:1ivc}:
In view of \ref{p:1iii},
$\ran(\proxcc{L}{B})=\dom(\proxc{L}{B^{-1}})
\subset L^*(\dom B^{-1})=L^*(\ran B)$. 

\ref{p:1v}: Combine Lemma~\ref{l:1}\ref{l:1ii} and
Proposition~\ref{p:0}.

\ref{p:1vii}: Let $x\in\HH$. With the help of 
Lemma~\ref{l:1}\ref{l:1i}--\ref{l:1ii} and
Proposition~\ref{p:0}, we derive that
\begin{eqnarray}
\label{e:d7}
x\in L^{-1}(\zer B)
&\Leftrightarrow& 0\in Lx-J_B(Lx)
\nonumber\\
&\Rightarrow& 0\in L^*\big((\Id_{\GG}-J_{B})Lx\big)
\nonumber\\
&\Leftrightarrow& 0\in L^*\big(J_{B^{-1}}Lx\big)
\nonumber\\
&\Leftrightarrow& 0\in J_{\proxc{L}{B^{-1}}}x
\nonumber\\
&\Leftrightarrow& x\in\big(\Id_{\GG}-J_{\proxc{L}{B^{-1}}}\big)x
\nonumber\\
&\Leftrightarrow& x\in J_{(\proxc{L}{B^{-1}})^{-1}}x
\nonumber\\
&\Leftrightarrow& x\in\zer\big(\proxcc{L}{B}\big).
\end{eqnarray}

\ref{p:1vi}: It follows from Lemma~\ref{l:1}\ref{l:1i} that
$(\proxc{L}{B})\infconv\Id_{\HH}+
(\proxc{L}{B})^{-1}\infconv\Id_{\HH}=\Id_{\HH}$. On the other
hand, Proposition~\ref{p:0} yields
$(\proxc{L}{B})^{-1}\infconv\Id_{\HH}=J_{\proxc{L}{B}}=
L^*\circ(B^{-1}\infconv\Id_{\GG})\circ L$.

\ref{p:1x}:
It follows from \eqref{e:r}, \ref{p:1ii}, and
Proposition~\ref{p:0} that
$(\proxcc{L}{B})\infconv\Id_{\HH}=J_{(\proxcc{L}{B})^{-1}}
=J_{\proxc{L}{B^{-1}}}=L^*\circ J_{B^{-1}}\circ L
=L^*\circ(B\infconv\Id_\GG)\circ L$.
\end{proof}

\begin{remark}[isometry] 
\label{r:2} 
In connection with Proposition~\ref{p:1}\ref{p:1iic}, here are some
important settings in which $L$ is an isometry:
\begin{enumerate}
\item 
\label{r:2i}
Example~\ref{ex:5} under the assumption that
$\sum_{k=1}^p\omega_kL_k^*\circ L_k=\Id_\HH$. 
\item
\label{r:2ii}
The resolvent average of Example~\ref{ex:1}, as a realization of
\ref{r:2i}.
\item 
\label{r:2iii}
Example~\ref{ex:11} under the assumption that $(e_k)_{k\in\NN}$ 
is a \emph{Parseval frame}, i.e., $\alpha=\beta=1$ in \eqref{e:f}. 
\end{enumerate} 
\end{remark}

\begin{proposition}
\label{p:9}
Let $\KK$ be a real Hilbert space, let $Q\in\BL(\HH,\GG)$, 
let $L\in\BL(\GG,\KK)$, and let $B\colon\KK\to 2^{\KK}$. Then
$\proxc{Q}{(\proxc{L}{B})}=\proxc{(L\circ Q)}{B}$.
\end{proposition}
\begin{proof}
It follows from Proposition~\ref{p:1}\ref{p:1i} and 
Proposition~\ref{p:0} that
$\proxc{Q}{(\proxc{L}{B})}
=(Q^*\circ J_{\proxc{L}{B}}\circ Q)^{-1}-\Id_\HH
=(Q^*\circ L^*\circ J_{B}\circ L\circ Q)^{-1}-\Id_\HH
=((L\circ Q)^*\circ J_{B}\circ(L\circ Q))^{-1}-\Id_\HH
=\proxc{(L\circ Q)}{B}$.
\end{proof}

The next results bring into play monotonicity. A key fact is that,
if $L$ is nonexpansive, then the resolvent composition preserves
monotonicity and maximal monotonicity.

\begin{proposition}
\label{p:2}
Let $L\in\BL(\HH,\GG)$ and let $B\colon\GG\to 2^{\GG}$ be monotone.
Then the following hold:
\begin{enumerate}
\item
\label{p:2i}
Suppose that $\|L\|\leq 1$. Then $\proxc{L}{B}$ is monotone.
\item
\label{p:2ic}
Suppose that $\|L\|\leq 1$. Then $\proxcc{L}{B}$ is monotone.
\item
\label{p:2ii}
Suppose that $L\neq 0$, let $\alpha\in\RP$ be such that
$B-\alpha\Id_\GG$ is monotone, set 
$\beta=(\alpha+1)\|L\|^{-2}-1$, and suppose that one of the
following is satisfied:
\begin{enumerate}
\item
\label{p:2ia-}
$\|L\|<\sqrt{\alpha+1}$.
\item
\label{p:2ia}
$\|L\|\leq 1$ and $\alpha>0$, i.e., $B$ is $\alpha$-strongly
monotone. 
\item
\label{p:2ib}
$\|L\|<1$.
\end{enumerate}
Then $\proxc{L}{B}$ is $\beta$-strongly monotone.
\end{enumerate}
\end{proposition}
\begin{proof}
\ref{p:2i}:
We set $R=L^*\circ J_B\circ L$ and note that $R$ is single-valued
on its domain since Lemma~\ref{l:0}\ref{l:0i}-\ref{l:0ii} states
that $J_B$ is. Now take $(x_1,x_1^*)\in\gra(\proxc{L}{B})$ and
$(x_2,x_2^*)\in\gra(\proxc{L}{B})$. By
Proposition~\ref{p:1}\ref{p:1i+}, $(x_1+x_1^*,x_1)\in\gra R$ and
$(x_2+x_2^*,x_2)\in\gra R$, i.e., $x_1=R(x_1+x_1^*)$ and
$x_2=R(x_2+x_2^*)$. However, since $R$ is firmly nonexpansive by
Lemma~\ref{l:9}\ref{l:9ii}, we get
\begin{align}
\scal{x_1-x_2}{x_1^*-x_2^*}
&=\scal{R(x_1+x_1^*)-R(x_2+x_2^*)}{(x_1+x_1^*)-(x_2+x_2^*)}
-\|x_1-x_2\|^2\nonumber\\
&\geq\|R(x_1+x_1^*)-R(x_2+x_2^*)\|^2-\|x_1-x_2\|^2\nonumber\\
&=0,
\end{align}
which establishes \eqref{e:mon2}.

\ref{p:2ic}: Since monotonicity is preserved under inversion,
$B^{-1}$ is monotone, and so is $\proxc{L}{B^{-1}}$ by \ref{p:2i}.
In turn, if $\proxcc{L}{B}=(\proxc{L}{B^{-1}})^{-1}$ is monotone as
well.

\ref{p:2ii}: We consider only property \ref{p:2ia-}, which implies
that $\beta>0$, since \ref{p:2ia} and \ref{p:2ib} are special cases
of it. In view of Lemma~\ref{l:0}\ref{l:0ii} (for $\alpha=0$) and
Lemma~\ref{l:1}\ref{l:1iv} (for $\alpha>0$), $J_B$ is
$(\alpha+1)$-cocoercive and $L^*\circ J_B\circ L$ is therefore
$(\alpha+1)\|L\|^{-2}$-cocoercive on account of 
Lemma~\ref{l:9}\ref{l:9i}.
This shows that $(L^*\circ J_B\circ L)^{-1}$ is 
$(\alpha+1)\|L\|^{-2}$-strongly monotone. Appealing to
Proposition~\ref{p:1}\ref{p:1i}, we conclude that
$\proxc{L}{B}=(L^*\circ J_B\circ L)^{-1}-\Id_{\HH}$ is
$\beta$-strongly monotone.
\end{proof}

The theorem below significantly improves
Proposition~\ref{p:2}\ref{p:2i}-\ref{p:2ic} and
Proposition~\ref{p:1}\ref{p:1iii}--\ref{p:1ivc} in the case of
maximally monotone operators. 

\begin{theorem}
\label{t:1}
Let $L\in\BL(\HH,\GG)$ be such that $\|L\|\leq 1$ and let
$B\colon\GG\to 2^{\GG}$ be maximally monotone. Then the following
hold:
\begin{enumerate}
\item
\label{t:1i}
$\proxc{L}{B}$ is maximally monotone.
\item
\label{t:1ic}
$\proxcc{L}{B}$ is maximally monotone.
\item
\label{t:1i+}
Suppose that $L$ is injective and that $B$ is at most
single-valued. Then $\proxc{L}{B}$ is at most single-valued.
\item
\label{t:1i++}
Suppose that $L$ and $B$ are injective. Then $\proxc{L}{B}$ is
injective.
\item
\label{t:1ii}
$\inte\dom(\proxc{L}{B})=\inte L^*(\dom B)$.
\item
\label{t:1iii}
$\cdom(\proxc{L}{B})=\overline{L^*(\dom B)}$.
\item
\label{t:1iv}
$\inte\ran(\proxc{L}{B})=\inte(\ran(\Id_\HH-L^*\circ L)
+L^*(\ran B))$.
\item
\label{t:1v}
$\cran(\proxc{L}{B})=\overline{\ran(\Id_\HH-L^*\circ L)
+L^*(\ran B)}$.
\item
\label{t:1iic}
$\inte\dom(\proxcc{L}{B})=\inte(\ran(\Id_\HH-L^*\circ L)
+L^*(\dom B))$.
\item
\label{t:1iiic}
$\cdom(\proxcc{L}{B})=\overline{\ran(\Id_\HH-L^*\circ L)
+L^*(\dom B)}$.
\item
\label{t:1ivc}
$\inte\ran(\proxcc{L}{B})=\inte L^*(\ran B)$.
\item
\label{t:1vc}
$\cran(\proxcc{L}{B})=\overline{L^*(\ran B)}$.
\end{enumerate}
\end{theorem}
\begin{proof}
It follows from Lemma~\ref{l:0}\ref{l:0iii} that 
$J_B\colon\GG\to\GG$ is firmly nonexpansive. Hence, we derive from
Lemma~\ref{l:9}\ref{l:9iii} that 
\begin{equation}
\label{e:26}
L^*\circ J_B\circ L\;\text{is maximally monotone.}
\end{equation}

\ref{t:1i}: It follows from \eqref{e:26} that 
$(L^*\circ J_B\circ L)^{-1}$ is maximally monotone. In view of
Proposition~\ref{p:1}\ref{p:1i}, Proposition~\ref{p:2}\ref{p:2i},
and Lemma~\ref{l:2}, we conclude that $\proxc{L}{B}$ is maximally
monotone.

\ref{t:1ic}: Since maximal monotonicity is preserved under
inversion, $B^{-1}$ is maximally monotone. In view of \ref{t:1i},
this renders $\proxc{L}{B^{-1}}$ maximally monotone. We then infer
that $\proxcc{L}{B}=(\proxc{L}{B^{-1}})^{-1}$ is maximally
monotone.

\ref{t:1i+}: Let us first recall that a maximally monotone operator
is at most single-valued if and only if its resolvent is injective
\cite[Theorem~2.1(iv)]{Baus12}. Hence, $J_B$ is injective and,
appealing to \ref{t:1i} and Proposition~\ref{p:0}, it is enough to
show that $L^*\circ J_B\circ L$ is injective. Let $x_1\in\HH$ and
$x_2\in\HH$ be such that 
$(L^*\circ J_B\circ L)x_1=(L^*\circ J_B\circ L)x_2$. Then, since
Lemma~\ref{l:0}\ref{l:0iii} asserts that $J_B$ is firmly
nonexpansive, 
\begin{align}
0&
=\scal{(L^*\circ J_B\circ L)x_1-(L^*\circ J_B\circ L)x_2}{x_1-x_2}
\nonumber\\
&=\scal{J_B(Lx_1)-J_B(Lx_2)}{Lx_1-Lx_2}\nonumber\\
&\geq\|J_B(Lx_1)-J_B(Lx_2)\|^2.
\end{align}
Therefore $J_B(Lx_1)=J_B(Lx_2)$ and, since $J_B$ is injective,
$Lx_1=Lx_2$.
Finally, the injectivity of $L$ yields $x_1=x_2$.

\ref{t:1i++}: Using the fact that a maximally monotone operator is
injective if and only if its resolvent is strictly nonexpansive
\cite[Theorem~2.1(ix)]{Baus12}, we obtain the strict
nonexpansiveness of $J_B$. Furthermore, according to \ref{t:1i} 
and Proposition~\ref{p:0}, it is enough to show that 
$L^*\circ J_B\circ L$ is strictly nonexpansive. To this end, we let
$x_1\in\HH$ and $x_2\in\HH$ be such that 
\begin{equation}
\|(L^*\circ J_B\circ L)x_1-(L^*\circ J_B\circ L)x_2\|=\|x_1-x_2\|. 
\end{equation}
Then, since $\|L^*\|=\|L\|\leq 1$, 
\begin{align}
\|x_1-x_2\|
&=\|(L^*\circ J_B\circ L)x_1-(L^*\circ J_B\circ L)x_2\|\nonumber\\
&\leq\|J_B(Lx_1)-J_B(Lx_2)\|\nonumber\\
&\leq\|Lx_1-Lx_2\|\nonumber\\
&\leq\|x_1-x_2\|.
\end{align}
Thus, $\|J_B(Lx_1)-J_B(Lx_2)\|=\|Lx_1-Lx_2\|$ and, since $J_B$ is
strictly nonexpansive, we obtain $Lx_1=Lx_2$. In view of the
injectivity of $L$, this means that $x_1=x_2$. As
Lemma~\ref{l:0}\ref{l:0iii} and Lemma~\ref{l:9}\ref{l:9ii} imply
that $L^*\circ J_B\circ L$ is nonexpansive, we conclude that it is
strictly nonexpansive.

\ref{t:1ii}--\ref{t:1iii}: 
Arguing as in \eqref{e:h4}, we observe that 
\begin{equation}
\label{e:h7}
\ran(L^*\circ J_B\circ L)=\dom(\proxc{L}{B})\subset L^*(\dom B).
\end{equation}
On the other hand, \cite[Example~25.20(ii)]{Livre1} asserts that
$J_B$ is $3^*$~monotone. Therefore, we derive from \eqref{e:26}
and Lemma~\ref{l:3}\ref{l:3i} that  
\begin{equation}
\inte L^*(\dom B)=\inte L^*(\ran J_B)
\subset\ran(L^*\circ J_B\circ L)=\dom(\proxc{L}{B})
\subset L^*(\dom B),
\end{equation}
which yields \ref{t:1ii}. Let us turn to \ref{t:1iii}.
Proceeding as above and invoking
Lemma~\ref{l:3}\ref{l:3ii}, \eqref{e:h7} yields
\begin{equation}
{L^*(\dom B)}={L^*(\ran J_B)}
\subset\cran(L^*\circ J_B\circ L)=\cdom(\proxc{L}{B})
\subset\overline{L^*(\dom B)}
\end{equation}
and, therefore, $\cdom(\proxc{L}{B})=\overline{L^*(\dom B)}$.

\ref{t:1iv}--\ref{t:1v}: Set 
\begin{equation}
\begin{cases}
A=\Id_\HH-L^*\circ L\\
\boldsymbol{L}\colon\HH\to\HH\oplus\GG\colon x\mapsto(x,Lx)\\
\boldsymbol{B}\colon\HH\oplus\GG\to 2^{\HH}\times 2^{\GG}
\colon(x,y)\mapsto Ax\times\{J_{B^{-1}}y\}.
\end{cases}
\end{equation}
Since $\boldsymbol{L}^*\colon\HH\oplus\GG\to\HH\colon
(x^*,y^*)\mapsto x^*+L^*y^*$, we deduce from \eqref{e:h24} 
and \eqref{e:h5} that 
\begin{equation}
\label{e:h6}
\ran(\boldsymbol{L}^*\circ\boldsymbol{B}\circ\boldsymbol{L})=
\ran(A+L^*\circ J_{B^{-1}}\circ L)
=\ran(\proxc{L}{B})\subset\ran A +L^*(\ran B)
=\boldsymbol{L}^*(\ran\boldsymbol{B}).
\end{equation}
In addition, since
\begin{equation}
(\forall x\in\HH)\quad\scal{x}{L^*(Lx)}=\|Lx\|^2\geq
\|L\|^2\,\|Lx\|^2\geq\|L^*(Lx)\|^2,
\end{equation}
the operator $L^*\circ L$ is firmly nonexpansive and so is
therefore $A=\Id_\HH-L^*\circ L$, which is thus maximally monotone
by virtue of \cite[Example~20.30]{Livre1}. In view of 
\cite[Proposition~25.16]{Livre1}, this means that $A$ is 
$3^*$~monotone. On the other hand, since $B^{-1}$ is maximally
monotone, we derive from \cite[Example~25.20(iii)]{Livre1} that
$J_{B^{-1}}$ is $3^*$~monotone. Thus, $\boldsymbol{B}$ 
is $3^*$~monotone. Moreover, since \cite[Proposition~20.23]{Livre1}
implies that $\boldsymbol{B}$ is maximally monotone and since
$\dom\boldsymbol{B}=\HH\oplus\GG$, it follows from
\cite[Corollary~25.6]{Livre1} that
$\boldsymbol{L}^*\circ\boldsymbol{B}\circ\boldsymbol{L}$ is
maximally monotone. We can therefore invoke
Lemma~\ref{l:3}\ref{l:3i} to obtain
\begin{equation}
\label{e:h8}
\inte\boldsymbol{L}^*(\ran\boldsymbol{B})\subset
\ran(\boldsymbol{L}^*\circ\boldsymbol{B}\circ\boldsymbol{L}).
\end{equation}
In view of \eqref{e:h6}, this proves \ref{t:1iv}. Similarly, 
Lemma~\ref{l:3}\ref{l:3ii} guarantees that
\begin{equation}
\label{e:h9}
\boldsymbol{L}^*(\ran\boldsymbol{B})\subset
\cran(\boldsymbol{L}^*\circ\boldsymbol{B}\circ\boldsymbol{L})
\end{equation}
and, using \eqref{e:h6}, we arrive at \ref{t:1v}.

\ref{t:1iic}: Using \ref{t:1iv}, we obtain
\begin{align}
\inte\dom\proxcc{L}{B}
&=\inte\ran\proxc{L}{B^{-1}}\nonumber\\
&=\inte\big(\ran(\Id_\HH-L^*\circ L)+L^*(\ran B^{-1})\big)
\nonumber\\
&=\inte\big(\ran(\Id_\HH-L^*\circ L)+L^*(\dom B)\big).
\end{align}

\ref{t:1iiic}: Using \ref{t:1v}, we obtain
\begin{align}
\cdom\proxcc{L}{B}
&=\cran\proxc{L}{B^{-1}}\nonumber\\
&=\overline{\ran(\Id_\HH-L^*\circ L)+L^*(\ran B^{-1})}
\nonumber\\
&=\overline{\ran(\Id_\HH-L^*\circ L)+L^*(\dom B)}.
\end{align}

\ref{t:1ivc}: Using \ref{t:1ii}, we obtain
$\inte\ran\proxcc{L}{B}=\inte\dom\proxc{L}{B^{-1}}
=\inte L^*(\dom B^{-1})=\inte L^*(\ran B)$.

\ref{t:1vc}: Using \ref{t:1iii}, we obtain
$\cran\proxcc{L}{B}=\cdom\proxc{L}{B^{-1}}
=\overline{L^*(\dom B^{-1})}=\overline{L^*(\ran B)}$.
\end{proof}

\begin{corollary}
\label{c:3}
Suppose that $L\in\BL(\HH,\GG)$ satisfies $\|L\|\leq 1$ and let 
$B\colon\GG\to 2^{\GG}$ be maximally monotone. Then the following 
hold:
\begin{enumerate}
\item
\label{c:3i}
Suppose that $L^*(\dom B)=\HH$. Then $\dom(\proxc{L}{B})=\HH$. 
\item
\label{c:3ii}
Suppose that $\ran(\Id_\HH-L^*\circ L)+L^*(\ran B)=\HH$. Then
$\proxc{L}{B}$ is surjective.
\item
\label{c:3ic}
Suppose that $\ran(\Id_\HH-L^*\circ L)+L^*(\dom B)=\HH$. Then
$\dom(\proxcc{L}{B})=\HH$.
\item
\label{c:3iii}
Suppose that $L^*(\ran B)=\HH$. 
Then $\proxcc{L}{B}$ is surjective.
\end{enumerate}
\end{corollary}
\begin{proof}
We deduce \ref{c:3i} from Theorem~\ref{t:1}\ref{t:1ii}, 
\ref{c:3ii} from Theorem~\ref{t:1}\ref{t:1iv}, 
\ref{c:3ic} from Theorem~\ref{t:1}\ref{t:1iic}, 
and \ref{c:3iii} from Theorem~\ref{t:1}\ref{t:1ivc}.
\end{proof}

\begin{example}
\label{ex:4}
Going back to Example~\ref{ex:9}, let $B\colon\HH\to 2^{\HH}$ be
maximally monotone and suppose that $V\neq\{0\}$ is a closed vector
subspace of $\HH$ such that $(\forall v\in V)$ 
$(v+V^\bot)\cap\ran B\neq\emp$. Then $\proxc{\proj_V}{B}$ is
surjective.
\end{example}
\begin{proof}
Set $L=\proj_V$. Then $\|L\|=1$ and
$\ran(\Id_\HH-L^*\circ L)=\ran(\Id_\HH-\proj_V)=V^\bot$.
On the other hand, $(\forall v\in V)(\exi x^*\in\ran B)$ 
$x^*\in v+V^\bot=\proj_V^{-1} v$. Therefore 
$L^*(\ran B)=\proj_V(\ran B)=V$. Thus,
$\ran(\Id_\HH-L^*\circ L)+L^*(\ran B)=V+V^\bot=\HH$ and the result
follows from Corollary~\ref{c:3}\ref{c:3ii}.
\end{proof}

\begin{proposition}
\label{p:7}
Suppose that $L\in\BL(\HH,\GG)$, let $\beta\in\RPP$, let $D$
be a nonempty subset of $\GG$, let $B\colon D\to\GG$ be
$\beta$-cocoercive, suppose that $0<\|L\|<\sqrt{\beta+1}$, and set
$\alpha=(\beta+1)\|L\|^{-2}-1$. Then $\proxcc{L}{B}$ is
$\alpha$-cocoercive.
\end{proposition}
\begin{proof}
Since $B^{-1}$ is $\beta$-strongly monotone, 
Lemma~\ref{l:1}\ref{l:1iv} entails that $J_{B^{-1}}$ is
$(\beta+1)$-cocoercive. In turn, by Lemma~\ref{l:9}\ref{l:9i},
$L^*\circ J_{B^{-1}}\circ L$ is $(\beta+1)\|L\|^{-2}$-cocoercive,
which makes $(L^*\circ J_{B^{-1}}\circ L)^{-1}$ a
$(\beta+1)\|L\|^{-2}$-strongly monotone operator. In view of 
Proposition~\ref{p:1}\ref{p:1ii} and 
Proposition~\ref{p:1}\ref{p:1i}, we conclude that 
\begin{equation} 
(\proxcc{L}{B})^{-1}=\proxc{L}{B^{-1}}=
\big(L^*\circ J_{B^{-1}}\circ L\big)^{-1}-\Id_\HH
\end{equation} 
is $\alpha$-strongly monotone and hence that $\proxcc{L}{B}$ is
$\alpha$-cocoercive. 
\end{proof}

\begin{proposition}
\label{p:8}
Let $L\in\BL(\HH,\GG)$ be such that $\|L\|\leq 1$,
let $D$ be a nonempty subset of $\GG$, and let
$B\colon D\to\GG$ be monotone and nonexpansive.
Then $\proxcc{L}{B}$ is monotone and nonexpansive.
\end{proposition}
\begin{proof}
The monotonicity of $\proxcc{L}{B}$ is established in
Proposition~\ref{p:2}\ref{p:2ic}. Let us show its nonexpansiveness.
Since $B$ is nonexpansive, it follows from
\cite[Proposition~4.4]{Livre1} and Lemma~\ref{l:0}\ref{l:0ii} that
there exists a monotone operator $E\colon\GG\to 2^{\GG}$ such that
$B=2J_E-\Id_\GG$. Now set $M=\Id_\HH-L^*\circ L+L^*\circ E\circ L$.
Since $\|L\|\leq 1$, $\Id_\HH-L^*\circ L$ is monotone, while
$L^*\circ E\circ L$ is monotone by
\cite[Proposition~20.10]{Livre1}. The sum $M$ of these two
operators is therefore monotone, which renders $J_M$ firmly
nonexpansive by Lemma~\ref{l:0}\ref{l:0ii}, and hence
$2J_M-\Id_\HH$ nonexpansive. On the other hand,
Proposition~\ref{p:1}\ref{p:1ic2} yields 
\begin{align}
J_{\proxcc{L}{B}}
&=\Id_\HH-L^*\circ L+L^*\circ(B+\Id_\GG)^{-1}\circ L
\nonumber\\
&=\Id_\HH-L^*\circ L+L^*\circ(2J_E)^{-1}\circ L
\nonumber\\
&=\big(2\Id_\HH-2L^*\circ L+L^*\circ(E+\Id_\GG)\circ L\big)
\circ(\Id_\HH/2)\nonumber\\
&=(\Id_\HH+M)\circ(\Id_\HH/2)
\nonumber\\
&=(2J_M)^{-1}.
\end{align}
We have thus verified that $\proxcc{L}{B}=2J_M-\Id_\HH$ is
nonexpansive.
\end{proof}

\begin{remark}[resolvent average]
\label{r:3}
Consider the setting of Example~\ref{ex:1}, where 
$\sum_{k=1}^p\omega_k=1$, and let $A$ be the resolvent average of 
the operators $(B_k)_{1\leq k\leq p}$ defined in 
\eqref{e:4}. Then, as discussed in Example~\ref{ex:1},
Remark~\ref{r:2}\ref{r:2ii}, and Proposition~\ref{p:1}\ref{p:1iic},
$A=\proxc{L}{B}=\proxcc{L}{B}$, where 
$L\colon x\mapsto (x,\ldots,x)$ is an isometry with adjoint
$L^*\colon(y_k)_{1\leq k\leq p}\mapsto\sum_{k=1}^p\omega_ky_k$ and
$B\colon(y_k)_{1\leq k\leq p}\mapsto 
B_1y_1\times\cdots\times B_py_p$. We can therefore establish at
once from the above results various properties of the resolvent
average, such as the following:
\begin{enumerate}
\item
Proposition~\ref{p:1}\ref{p:1ii} yields
$A^{-1}=\proxc{L}{B^{-1}}=(\sum_{k=1}^p\omega_k
(B_k^{-1}+\Id_\HH)^{-1})^{-1}-\Id_{\HH}$ 
(see \cite[Theorem~2.2]{Baus16}).
\item
Suppose that the operators $(B_k)_{1\leq k\leq p}$ are 
monotone. Then Theorem~\ref{t:1}\ref{t:1i} asserts that $A$ is
maximally monotone if the operators $(B_k)_{1\leq k\leq p}$ are. 
In addition, Proposition~\ref{p:1}\ref{p:1iii} asserts that 
$\dom A\subset\sum_{k=1}^p\omega_k\dom B_k$ and
Proposition~\ref{p:1}\ref{p:1ivc} that
$\ran A\subset\sum_{k=1}^p\omega_k\ran B_k$ 
(see \cite[Proposition~2.7]{Baus16} and note that maximality is not
required in the last two properties).
\item
Suppose that the operators $(B_k)_{1\leq k\leq p}$ are maximally
monotone. Then Theorem~\ref{t:1}\ref{t:1ii}--\ref{t:1iii} yields
$\intdom A=\inte\sum_{k=1}^p\omega_k\,\dom B_k$ and
$\cdom A=\overline{\sum_{k=1}^p\omega_k\,\dom B_k}$, 
while Theorem~\ref{t:1}\ref{t:1ivc}--\ref{t:1vc} yields
$\inte\ran A=\inte\sum_{k=1}^p\omega_k\,\ran B_k$, and 
$\cran A=\overline{\sum_{k=1}^p\omega_k\,\ran B_k}$
(see \cite[Theorem~2.11]{Baus16}).
\item
\label{r:3iv}
Suppose that the operators $(B_k)_{1\leq k\leq p}$ are maximally
monotone and strongly monotone. Then it follows from 
Proposition~\ref{p:2}\ref{p:2ia} that $A$ is strongly
monotone (see \cite[Theorem~3.20]{Baus16}, where the strong
monotonicity of $A$ is established under the more
general assumption that only one of the operators 
$(B_k)_{1\leq k\leq p}$ is strongly monotone).
\item
Suppose that, for every $k\in\{1,\ldots,p\}$,
$B_k\colon\GG_k\to\GG_k$ is monotone and nonexpansive. Then it
follows from Proposition~\ref{p:8} that $A$ is monotone and
nonexpansive (see \cite[Theorem~4.16]{Baus16}).
\end{enumerate}
\end{remark}

\begin{remark}[parametrization]
\label{r:8}
A parameter $\gamma\in\RPP$ can be introduced in 
Definition~\ref{d:1} by putting
\begin{equation}
\label{e:gb}
\proxcg{L}{B}=L^*\pushfwd(B+\gamma^{-1}\Id_{\GG})-
\gamma^{-1}\Id_{\HH}. 
\end{equation}
In the special case of the resolvent average discussed in 
Example~\ref{ex:1}, \eqref{e:gb} leads to the parametrized
version of \eqref{e:4} considered in \cite{Baus16}, namely
$\proxcg{L}{B}=(\sum_{k=1}^p\omega_k(B_k+\gamma^{-1}
\Id_\HH)^{-1})^{-1}-\gamma^{-1}\Id_{\HH}$. In general, 
with the assistance of Lemma~\ref{l:1}\ref{l:1i-} and
Proposition~\ref{p:0}, we obtain 
\begin{equation}
J_{\gamma(\proxcg{L}{B})}=L^*\circ J_{\gamma B}\circ L
=J_{\proxc{L}{(\gamma B)}}. 
\end{equation}
This shows that the
parametrized version \eqref{e:gb} is closely related to the
original one \eqref{e:2} since
$\gamma(\proxcg{L}{B})=\proxc{L}{(\gamma B)}$.
The proximal composition of Definition~\ref{d:2} can be 
parametrized similarly by putting
$\proxcg{L}{g}=((g^*\infconv(\gamma\qq_{\GG}))\circ L)^*
-\gamma^{-1}\qq_{\HH}$.
\end{remark}

\begin{remark}[warping]
\label{r:w}
An extension of Definition~\ref{d:1} can be devised using the
theory of warped resolvents \cite{Jmaa20}. Let $\XX$ and $\YY$ be
reflexive real Banach spaces, let $K_\YY\colon\YY\supset
D_\YY\to\YY^*$, let $L\in\BL(\XX,\YY)$, and let $B\colon\YY\to
2^{\YY^*}$. Then, under suitable conditions, the warped resolvent
of $B$ with kernel $K_\YY$ is $J_B^{K_\YY}=(B+K_\YY)^{-1}\circ
K_\YY$ (for instance, if $h\colon\YY\to\RX$ is a Legendre function
such that $\dom B\subset\intdom h$ and $K_\YY=\nabla h$, then
$J_B^{K_\YY}$ is the $D$-resolvent of $B$ \cite{Sico03}). For a
suitable kernel $K_\XX\colon\XX\supset D_\XX\to\XX^*$, we then
define the \emph{warped resolvent composition}
$\proxc{L}{B}=K_\XX\circ(L^*\pushfwd
(K_\YY^{-1}\circ(B+\KK_\YY)))-K_\XX$, which yields
$J^{K_\XX}_{\proxc{L}{B}}=L^*\circ J^{K_\YY}_B\circ L$.
\end{remark}

\section{The proximal composition}
\label{sec:5}

This section is dedicated to the study of some aspects of the
proximal composition operations introduced in Definition~\ref{d:2}
and further discussed in Examples~\ref{ex:8} and~\ref{ex:7}.

\begin{remark}
\label{r:m}
The proximal composition was linked to the resolvent composition 
in Example~\ref{ex:8}\ref{ex:8ii}. We can also motivate this
construction via Moreau's theory of proximity operators and
envelopes \cite{Mor62b,Mor63a,More65}. Indeed, let
$g\in\Gamma_0(\GG)$, suppose that $L\in\BL(\HH,\GG)$ satisfies
$0<\|L\|\leq 1$, and set $T=L^*\circ\prox_g\circ L$. Then $T$ is
nonexpansive since $\prox_g$ and $L$ are. On the other hand, we
infer from Lemma~\ref{l:4}\ref{l:4i} that
$T=L^*\circ\nabla(g^*\infconv\qq_\GG)\circ L
=\nabla((g^*\infconv\qq_\GG)\circ L)$. Altogether, Lemma~\ref{l:5}
implies that $T=\prox_{f}$, where 
$f=((g^*\infconv\qq_\GG)\circ L)^*-\qq_\HH$. 
The function $f$ is precisely the proximal
composition $\proxc{L}{g}$. Thus, up to an additive constant,
$\proxc{L}{g}$ is the function the proximity operator of which is
$L^*\circ\prox_g\circ L$.
\end{remark}

Let us now establish some properties of proximal compositions.

\begin{proposition}
\label{p:14}
Let $L\in\BL(\HH,\GG)$, let $g\colon\GG\to\RX$ and 
$h\colon\GG\to\RX$ be proper functions such that $h\leq g$, and 
let $\breve{g}$ be the largest lower semicontinuous convex
function majorized by $g$. Then the following hold:
\begin{enumerate}
\item
\label{p:14i}
$\proxc{L}{h}\leq\proxc{L}{g}$.
\item
\label{p:14ii}
Suppose that $h\geq\breve{g}$ and that $g$ admits a continuous 
affine minorant. Then $\proxc{L}{g}=\proxc{L}{h}$.
\item
\label{p:14iii}
Suppose that $g$ admits a continuous affine minorant. Then 
$\proxc{L}{g}=\proxc{L}{g^{**}}$.
\end{enumerate}
\end{proposition}
\begin{proof}
\ref{p:14i}:
In view of \eqref{e:conj} and \eqref{e:infconv1}, $h^*\geq g^*$ 
and hence $h^*\infconv\qq_\GG\geq g^*\infconv\qq_\GG$. Thus,
$(h^*\infconv\qq_\GG)\circ L\geq(g^*\infconv\qq_\GG)\circ L$ and
therefore $((h^*\infconv\qq_\GG)\circ L)^*\leq
((g^*\infconv\qq_\GG)\circ L)^*$. Appealing to \eqref{e:12}, we
conclude that $\proxc{L}{h}\leq\proxc{L}{g}$.

\ref{p:14ii}: Let $a$ be a continuous affine minorant of $g$. Then
$\minf<a=\breve{a}\leq\breve{g}\leq g\not\equiv\pinf$ and 
$\breve{g}$ is therefore proper. In addition, 
$\breve{g}\leq h\leq g$. Hence, \cite[Proposition~13.16]{Livre1}
yields $h^*=g^*$ and the conclusion follows from \eqref{e:12}.

\ref{p:14iii}: Since $\breve{g}=g^{**}$
\cite[Proposition~13.45]{Livre1}, the assertion follows from
\ref{p:14ii}.
\end{proof}

\begin{proposition}
\label{p:12}
Suppose that $L\in\BL(\HH,\GG)$ satisfies $0<\|L\|\leq 1$ and let
$g\colon\GG\to\RX$ be a proper function that admits a
continuous affine minorant. Then the following hold:
\begin{enumerate}
\item
\label{p:12i}
$\proxc{L}{g}=L^*\epushfwd(g^{**}+\qq_{\GG})-\qq_{\HH}$.
\item
\label{p:12ii}
$\dom(\proxc{L}{g})=L^*(\dom g^{**})$.
\item
\label{p:12iii}
$(\proxc{L}{g})^*=(\qq_{\HH}-(g^*\infconv\qq_{\GG})\circ L)^*
-\qq_{\HH}$.
\item
\label{p:12iv}
$(\proxc{L}{g})^*=\proxcc{L}{g^*}$.
\item
\label{p:12v}
$(\proxcc{L}{g})^*=\proxc{L}{g^*}$.
\item
\label{p:12vii}
$(\proxc{L}{g})\infconv\qq_\HH+(\proxcc{L}{g^*})\infconv\qq_\HH
=\qq_\HH$.
\item
\label{p:12vi}
Suppose that $L$ is an isometry. Then 
$\proxc{L}{g}=\proxcc{L}{g}$.
\end{enumerate}
\end{proposition}
\begin{proof}
By Example~\ref{ex:8}\ref{ex:8i-} and Lemma~\ref{l:7}\ref{l:7i},
$g^*$ is in $\Gamma_0(\GG)$ and it admits a continuous affine
minorant. In turn, we deduce from Lemma~\ref{l:4}\ref{l:4i-} that
$g^*\infconv\qq_{\GG}\in\Gamma_0(\GG)$ and hence that
$(g^*\infconv\qq_{\GG})\circ L\in\Gamma_0(\HH)$. We then deduce
from Lemma~\ref{l:7}\ref{l:7ii} that
$((g^*\infconv\qq_{\GG})\circ L)^*\in\Gamma_0(\HH)$.

\ref{p:12i}: Since $\dom(g^*\infconv\qq_{\GG})=\GG$ and
$g^*\infconv\qq_{\GG}\in\Gamma_0(\GG)$, it follows 
from \cite[Corollary~15.28(i)]{Livre1} and
Lemma~\ref{l:4}\ref{l:4i+} that $\proxc{L}{g}+\qq_{\HH}
=((g^*\infconv\qq_{\GG})\circ L)^*
=L^*\epushfwd(g^*\infconv \qq_{\GG})^*= 
L^*\epushfwd(g^{**}+\qq_{\GG})$.

\ref{p:12ii}: We invoke \ref{p:12i} and 
\cite[Proposition~12.36(i)]{Livre1} to get 
$\dom(\proxc{L}{g})=\dom(L^*\epushfwd(g^{**}+\qq_{\GG}))=
L^*(\dom(g^{**}+\qq_{\GG}))=L^*(\dom g^{**})$.

\ref{p:12iii}:
Since $((g^*\infconv\qq_{\GG})\circ L)^*\in\Gamma_0(\HH)$,
it follows from Definition~\ref{d:2} and
\cite[Proposition~13.29]{Livre1} that
\begin{align}
(\proxc{L}{g})^*
&=\Big(\big((g^*\infconv\qq_{\GG})\circ L\big)^*-
\qq_{\HH}\Big)^*\nonumber\\
&=\Big(\qq_{\HH}-\big((g^*\infconv\qq_{\GG})
\circ L\big)^{**}\Big)^*-\qq_{\HH}
\nonumber\\
&=\big(\qq_{\HH}-(g^*\infconv\qq_{\GG})\circ L\big)^*-\qq_{\HH}.
\end{align}

\ref{p:12iv}:
Proposition~\ref{p:14}\ref{p:14iii} yields 
$\proxcc{L}{g^*}=(\proxc{L}{g^{**}})^*=(\proxc{L}{g})^*$.

\ref{p:12v}: Example~\ref{ex:8}\ref{ex:8i-}--\ref{ex:8i} implies
that $\proxc{L}{g^*}\in\Gamma_0(\GG)$. In turn,
Lemma~\ref{l:7}\ref{l:7ii} yields 
$(\proxcc{L}{g})^*=(\proxc{L}{g^*})^{**}=\proxc{L}{g^*}$.

\ref{p:12vii}: Combine Example~\ref{ex:8}\ref{ex:8i}, 
Lemma~\ref{l:4}\ref{l:4ii}, and \ref{p:12iv}.

\ref{p:12vi}: Since $\qq_\HH=\qq_\GG\circ L$, we derive from 
Lemma~\ref{l:4}\ref{l:4ii} and \ref{p:12iii} that 
\begin{align}
\proxc{L}{g}
&=\big((g^*\infconv\qq_{\GG})\circ L\big)^*-\qq_{\HH},
\nonumber\\
&=\big((\qq_{\GG}-g^{**}\infconv\qq_{\GG})\circ L\big)^*-\qq_{\HH}
\nonumber\\
&=\big(\qq_{\HH}-(g^{**}\infconv\qq_{\GG})\circ L\big)^*-\qq_{\HH}
\nonumber\\
&=(\proxc{L}{g^*})^*\nonumber\\
&=\proxcc{L}{g},
\end{align}
as claimed. 
\end{proof}

The next result concerns the case when $L$ is an isometry.

\begin{proposition}
\label{p:13}
Suppose that $L\in\BL(\HH,\GG)$ is an isometry and let
$g\colon\GG\to\RX$ be a proper function that admits a continuous
affine minorant. Then 
$(g^*\circ L)^*\leq\proxc{L}{g}\leq g\circ L$.
\end{proposition}
\begin{proof}
We recall from Example~\ref{ex:8}\ref{ex:8i} that 
$\proxc{L}{g}\in\Gamma_0(\HH)$.
Fix $x\in\HH$ and recall that $g^{**}\leq g$ 
\cite[Proposition~13.16(i)]{Livre1}. By 
Proposition~\ref{p:12}\ref{p:12i},
\begin{equation}
\label{e:g67}
(\proxc{L}{g})(x)
=\inf_{\substack{y\in\GG\\ L^*y=x}}
\big(g^{**}(y)+\qq_{\GG}(y)\big)-\qq_{\HH}(x)
\leq\inf_{\substack{y\in\GG\\ L^*y=x}}
\big(g(y)+\qq_{\GG}(y)\big)-\qq_{\HH}(x).
\end{equation}
Now set $y=Lx$. Then $L^*y=L^*(Lx)=x$ and 
$\qq_{\GG}(Lx)=\qq_{\HH}(x)$. Therefore, \eqref{e:g67} yields
\begin{equation}
\label{e:g68}
(\proxc{L}{g})(x)\leq g(Lx)+\qq_{\GG}(Lx)-\qq_{\HH}(x)
=(g\circ L)(x),
\end{equation}
which provides the second inequality. To prove the first one, we
recall from Example~\ref{ex:8}\ref{ex:8i-} that 
$g^*\in\Gamma_0(\GG)$. Therefore, $g^*$ admits a continuous
affine minorant by Lemma~\ref{l:7}\ref{l:7i}. In turn,
\eqref{e:g68} yields $\proxc{L}{g^*}\leq g^*\circ L$ and hence
$(\proxc{L}{g^*})^*\geq(g^*\circ L)^*$. We then invoke 
successively Proposition~\ref{p:14}\ref{p:14i},
Proposition~\ref{p:12}\ref{p:12iv}, and
Proposition~\ref{p:12}\ref{p:12vi} to obtain
\begin{equation}
\proxc{L}{g}\geq\proxc{L}{g^{**}}=\big(\proxc{L}{g^*}\big)^*\geq
\big(g^*\circ L\big)^*,
\end{equation}
as announced.
\end{proof}

Let us take a closer look at the proximal composition for functions
in $\Gamma_0(\GG)$.

\begin{theorem}
\label{t:4}
Suppose that $L\in\BL(\HH,\GG)$ satisfies $0<\|L\|\leq 1$ and let
$g\in\Gamma_0(\GG)$. Then the following hold:
\begin{enumerate}
\item
\label{t:4i}
$\proxc{L}{g}=L^*\epushfwd(g+\qq_{\GG})-\qq_{\HH}$.
\item
\label{t:4ii}
$\dom(\proxc{L}{g})=L^*(\dom g)$.
\item
\label{t:4vi}
$\Argmin(\proxc{L}{g})=\Fix(L^*\circ\prox_g\circ L)$.
\item
\label{t:4iv}
$(\proxc{L}{g})\infconv\qq_{\HH}=\qq_{\HH}-
(g^{*}\infconv\qq_{\GG})\circ L$.
\item
\label{t:4vii}
$(\proxcc{L}{g})\infconv\qq_\HH=(g\infconv\qq_\GG)\circ L$.
\item
\label{t:4iii}
$L^{-1}(\Argmin g)\subset\Argmin(\proxcc{L}{g})
=\Argmin((g\infconv\qq_\GG)\circ L)$.
\end{enumerate}
\end{theorem}
\begin{proof}
We recall from Lemma~\ref{l:7}\ref{l:7i} that $g$ admits a
continuous affine minorant and from Example~\ref{ex:8}\ref{ex:8i}
that $\proxc{L}{g}\in\Gamma_0(\HH)$.

\ref{t:4i}--\ref{t:4ii}: These follow from 
Proposition~\ref{p:12}\ref{p:12i}--\ref{p:12ii} and
Lemma~\ref{l:7}\ref{l:7ii}.

\ref{t:4vi}: 
Example~\ref{ex:8}\ref{ex:8iii} and \eqref{e:fermat} 
yield $\Argmin{(\proxc{L}{g})}=\Fix\prox_{\proxc{L}{g}}=
\Fix(L^*\circ\prox_g\circ L)$. 

\ref{t:4iv}: It follows from \ref{t:4i} that
$(\proxc{L}{g})+\qq_\HH=L^*\epushfwd(g+\qq_\GG)$. Therefore,
using Example~\ref{ex:8}\ref{ex:8i}, Lemma~\ref{l:4}\ref{l:4i+},
and \cite[Proposition~13.24(iv)]{Livre1}, we derive that
\begin{equation}
(\proxc{L}{g})^*\infconv\qq_\HH=
\big((\proxc{L}{g})+\qq_\HH\big)^*=
\big(L^*\epushfwd(g+\qq_\GG)\big)^*=
(g+\qq_\GG)^*\circ L
=(g^*\infconv\qq_\GG)\circ L.
\end{equation}
Hence, it follows from Lemma~\ref{l:4}\ref{l:4ii} that 
\begin{equation}
(\proxc{L}{g})\infconv\qq_{\HH}+
(g^{*}\infconv\qq_{\GG})\circ L=
(\proxc{L}{g})\infconv\qq_{\HH}
+(\proxc{L}{g})^*\infconv\qq_{\HH}=\qq_{\HH}.
\end{equation}

\ref{t:4vii}:
We use Example~\ref{ex:7}\ref{ex:7i}, Lemma~\ref{l:4}\ref{l:4ii},
Proposition~\ref{p:12}\ref{p:12v}, \ref{t:4iv}, and 
Lemma~\ref{l:7}\ref{l:7ii} to obtain
\begin{align}
(\proxcc{L}{g})\infconv\qq_\HH
&=\qq_\HH-(\proxcc{L}{g})^*\infconv\qq_\HH\nonumber\\
&=\qq_\HH-(\proxc{L}{g^*})\infconv\qq_\HH\nonumber\\
&=\qq_\HH-\big(\qq_\HH-(g^{**}\infconv\qq_\GG)\circ L\big)
\nonumber\\
&=(g\infconv\qq_\GG)\circ L.
\end{align}

\ref{t:4iii}: 
We derive from \eqref{e:pierre},
Proposition~\ref{p:1}\ref{p:1vii} with $B=\partial g$, and
Example~\ref{ex:7}\ref{ex:7ii} that
\begin{equation}
L^{-1}(\Argmin g)
=L^{-1}(\zer\partial g)
\subset\zer(\proxcc{L}{\partial g})
=\zer\partial(\proxcc{L}{g})
=\Argmin(\proxcc{L}{g}).
\end{equation}
Next, since $\proxcc{L}{g}\in\Gamma_0(\HH)$ by
Example~\ref{ex:7}\ref{ex:7i}, \cite[Proposition~17.5]{Livre1}
and \ref{t:4vii} yield
$\Argmin(\proxcc{L}{g})=\Argmin((\proxcc{L}{g})\infconv\qq_{\HH})
=\Argmin((g\infconv\qq_\GG)\circ L)$.
\end{proof}

\begin{proposition}
\label{p:5}
Suppose that $L\in\BL(\HH,\GG)$ satisfies $0<\|L\|\leq 1$, let
$\alpha\in\RP$, let $g\in\Gamma_0(\GG)$ be such that
$g-\alpha\qq_{\GG}$ is convex, and set
$\beta=(\alpha+1)\|L\|^{-2}-1$. Suppose that one of the following
is satisfied:
\begin{enumerate}
\item
\label{p:5ii}
$\alpha>0$, i.e., $g$ is $\alpha$-strongly convex. 
\item
\label{p:5iii}
$\|L\|<1$.
\end{enumerate}
Then $\proxc{L}{g}$ is $\beta$-strongly convex.
\end{proposition}
\begin{proof}
By assumption, $g-\alpha\qq_{\GG}\in\Gamma_0(\GG)$ and hence, by
Lemma~\ref{l:4}\ref{l:4i--}, $\partial(g-\alpha\qq_{\GG})$ is
maximally monotone. 
However, by Lemma~\ref{l:4}\ref{l:4vi}, 
\begin{equation}
\partial g=\partial\big((g-\alpha\qq_{\GG})+\alpha\qq_{\GG}\big)
=\partial(g-\alpha\qq_{\GG})+\alpha\Id_{\GG}
\end{equation}
and therefore $\partial g-\alpha\Id_{\GG}=
\partial(g-\alpha\qq_{\GG})$ is monotone. 
Moreover, by \cite[Remark~3.5.3]{Zali02},
$\partial g$ is $\alpha$-strongly monotone in \ref{p:5ii}. 
Altogether, it follows from Example~\ref{ex:8}\ref{ex:8ii} and 
Proposition~\ref{p:2}\ref{p:2ii} that 
$\partial(\proxc{L}{g})=\proxc{L}{\partial g}$
is $\beta$-strongly monotone. Appealing to
\cite[Remark~3.5.3]{Zali02} again, we conclude that $\proxc{L}{g}$
is $\beta$-strongly convex.
\end{proof}

\begin{proposition}
\label{p:6}
Suppose that $L\in\BL(\HH,\GG)$ satisfies $0<\|L\|\leq 1$, let
$\alpha\in\RPP$, let $g\colon\GG\to\RR$ be convex and
differentiable, with a $\alpha^{-1}$-Lipschitzian gradient, and set
$\beta=(\alpha+1)\|L\|^{-2}-1$. Then $\proxcc{L}{g}$ is
differentiable on $\HH$ and its gradient is
$\beta^{-1}$-Lipschitzian.
\end{proposition}
\begin{proof}
We derive from Lemma~\ref{l:8} that $g^*$ is
$\alpha$-strongly convex. In turn, 
Proposition~\ref{p:12}\ref{p:12v} and 
Proposition~\ref{p:5}\ref{p:5ii} imply
that $(\proxcc{L}{g})^*=\proxc{L}{g^*}$ is 
$\beta$-strongly convex. Invoking Lemma~\ref{l:8} once more,
we obtain the assertion.
\end{proof}

The remainder of this section is devoted to examples of proximal
compositions.

\begin{example}[linear projection]
\label{ex:39}
Let $V$ be a closed vector subspace of $\HH$ and let
$g\colon\HH\to\RX$ be a proper function that admits a continuous
affine minorant. Then
$\proxc{\proj_V}{g}=\iota_V+\big(g^*+d_V^2/2\big)^*$.
\end{example}
\begin{proof}
Let $x\in\HH$. By Proposition~\ref{p:12}\ref{p:12ii}, 
$\dom(\proxc{\proj_V}{g})=\proj_V(\dom g^{**})\subset V$.
Therefore, if $x\notin V$, then $(\proxc{\proj_V}{g})(x)=\pinf$.
Now suppose that $x\in V$ and note that, by Pythagoras' identity,
$(\forall v\in V^\bot)$ $\qq_\HH(x-v)=\qq_\HH(x)+\qq_\HH(v)$.
Hence, using Proposition~\ref{p:12}\ref{p:12i} and basic
conjugation calculus \cite[Chapter~13]{Livre1}, we get
\begin{align}
\big(\proxc{\proj_V}{g}\big)(x)
&=\min_{\substack{y\in\HH\\ \proj_Vy=x}}
g^{**}(y)+\qq_\HH(y)-\qq_\HH(x)\nonumber\\
&=\min_{y\in x+V^\bot}g^{**}(y)+\qq_\HH(y)-\qq_\HH(x)\nonumber\\
&=\min_{v\in V^\bot}g^{**}(x-v)+\qq_\HH(x-v)-\qq_\HH(x)\nonumber\\
&=\min_{v\in\HH}g^{**}(x-v)+\iota_{V^\bot}(v)+\qq_\HH(v)\nonumber\\
&=\big(g^{**}\infconv(\iota_{V^\bot}+\qq_\HH)\big)(x)\nonumber\\
&=\big((g^*)^*\infconv(d_V^2/2)^*\big)(x)\nonumber\\
&=\big(g^*+d_V^2/2\big)^*(x),
\end{align}
which establishes the identity.
\end{proof}

\begin{example}[proximal mixture]
\label{ex:40}
Let $0\neq p\in\NN$ and, for every $k\in\{1,\ldots,p\}$, let 
$\GG_k$ be a real Hilbert space, let $L_{k}\in\BL(\HH,\GG_k)$, 
let $\omega_k\in\RPP$, and let $g_k\in\Gamma_0(\GG_k)$.
Suppose that $0<\sum_{k=1}^p\omega_k\|L_k\|^2\leq 1$ and let $\GG$ 
be the standard product vector space
$\GG_1\times\cdots\times\GG_p$, with generic element 
$\boldsymbol{y}=(y_k)_{1\leq k\leq p}$, and equipped 
with the scalar product 
$(\boldsymbol{y},\boldsymbol{y}')\mapsto\sum_{k=1}^p
\omega_k\scal{y_k}{y'_k}$. Set 
$L\colon\HH\to\GG\colon x\mapsto(L_kx)_{1\leq k\leq p}$
and $g\colon\GG\to\RX\colon\boldsymbol{y}\mapsto
\sum_{k=1}^p\omega_kg_k(y_k)$.
Then $\qq_{\GG}\colon\GG\to\RR\colon\boldsymbol{y}\mapsto
\sum_{k=1}^p\omega_k\qq_{\GG_k}(y_k)$,
$L^*\colon\GG\to\HH\colon\boldsymbol{y}\mapsto\sum_{k=1}^p
\omega_kL_k^*y_k$, $\prox_g\colon\GG\to\GG\colon\boldsymbol{y}
\mapsto(\prox_{g_k}y_k)_{1\leq k\leq p}$, and 
$g^*\colon\GG\to\RX\colon\boldsymbol{y}^*\mapsto
\sum_{k=1}^p\omega_kg_k^*(y_k^*)$. Thus, $g\in\Gamma_0(\GG)$,
$0<\|L\|\leq 1$, \eqref{e:12} produces the 
\emph{proximal mixture}
\begin{equation}
\label{e:q2}
\proxc{L}{g}=\Bigg(\sum_{k=1}^p\,\omega_k\big(g_k^*\infconv
\qq_{\GG_k}\big)\circ L_k\Bigg)^*-\qq_\HH,
\end{equation}
and Example~\ref{ex:8} yields
\begin{equation}
\label{e:p13}
\proxc{L}{g}\in\Gamma_0(\HH)\quad\text{and}\quad
\prox_{\proxc{L}{g}}=
\sum_{k=1}^m\omega_kL_k^*\circ\prox_{g_k}\circ L_k.
\end{equation}
In particular if, for every $k\in\{1,\ldots,p\}$, $\GG_k=\HH$ and
$L_k=\Id_{\HH}$, then \eqref{e:q2} is the \emph{proximal average}
\begin{equation}
\label{e:jj7}
\proxc{L}{g}=\Bigg(\sum_{k=1}^p\,\omega_k\big(g_k^*\infconv
\qq_{\HH}\big)\Bigg)^*-\qq_\HH,
\end{equation}
which has been studied in \cite{Baus08} (see also \cite{Luce10}
for illustrations and numerical aspects). The fact that 
$\sum_{k=1}^m\omega_k\prox_{g_k}$ is a proximity operator was 
first observed by Moreau \cite{Mor63a,More65} as a consequence of 
Lemma~\ref{l:5}.
\end{example}

\begin{remark}[proximal sum]
\label{r:7}
In Example~\ref{ex:40}, if $\sum_{k=1}^p\omega_k\|L_k\|^2>1$, the
proximal mixture \eqref{e:q2} may not be a function in
$\Gamma_0(\HH)$. In the case of \eqref{e:jj7} with $p=2$ and
$\omega_1=\omega_2=1$, conditions under which the 
\emph{proximal sum}
$\proxc{L}{g}=(g_1^*\infconv\qq_{\HH}+g_2^*\infconv\qq_{\HH})^*
-\qq_\HH$ is in $\Gamma_0(\HH)$ are provided in 
\cite{Buim19,Bord18,Zara74}. 
\end{remark}

\begin{remark}[proximal average]
\label{r:6}
As in Remark~\ref{r:3}, we can specialize the above results to
establish in a straightforward fashion various properties of the
proximal average \eqref{e:jj7}. In this context, we define $\GG$
and $g$ as in Example~\ref{ex:40} with $\GG_1=\cdots=\GG_p=\HH$ and
$\sum_{k=1}^p\omega_k=1$, and set $L\colon\HH\to\GG\colon
x\mapsto(x,\ldots,x)$. Then $L$ is an isometry and the resulting
proximal average $f=\proxc{L}{g}=\proxcc{L}{g}$ of \eqref{e:jj7}
(see Proposition~\ref{p:12}\ref{p:12vi}) possesses in
particular the following properties:
\begin{enumerate}
\item
Example~\ref{ex:8}\ref{ex:8i} yields 
$f\in\Gamma_0(\HH)$ (see \cite[Corollary~5.2]{Baus08}).
\item
Example~\ref{ex:8}\ref{ex:8iii} yields
$\prox_f=\sum_{k=1}^p\omega_k\prox_{g_k}$ 
(see \cite[Theorem~6.7]{Baus08}).
\item
Proposition~\ref{p:12}\ref{p:12v} yields
$f^*=(\sum_{k=1}^p\,\omega_k(g_k\infconv\qq_{\HH}))^*-\qq_\HH$
(see \cite[Theorem~5.1]{Baus08}).
\item
Proposition~\ref{p:13} yields 
$(\sum_{k=1}^p\omega_kg^*_k)^*\leq f\leq\sum_{k=1}^p\omega_kg_k$
(see \cite[Theorem~5.4]{Baus08}).
\item
Theorem~\ref{t:4}\ref{t:4ii} yields 
$\dom f=\sum_{k=1}^p\omega_k\dom g_k$
(see \cite[Theorem~4.6]{Baus08}).
\item
Theorem~\ref{t:4}\ref{t:4vii} yields 
$f\infconv\qq_{\HH}=\sum_{k=1}^p\omega_k(g_k\infconv\qq_{\HH})$
(see \cite[Theorem~6.2(i)]{Baus08}).
\item
Theorem~\ref{t:4}\ref{t:4iii} yields 
$\Argmin(f\infconv\qq_{\HH})=
\Argmin\sum_{k=1}^p\omega_k(g_k\infconv\qq_{\HH})$
(see \cite[Corollary~6.4]{Baus08}).
\item
\label{r:6viii}
Suppose that the functions $(g_k)_{1\leq k\leq p}$ are strongly
convex. Then it follows from Proposition~\ref{p:5}\ref{p:5ii} that
$f$ is strongly convex (see \cite[Corollary~3.23]{Baus16}, where
the strong convexity of $f$ is shown to hold more generally under 
the assumption that one of the functions $(g_k)_{1\leq k\leq p}$ is
strongly convex).
\end{enumerate}
\end{remark}

\section{Application to monotone inclusion models}
\label{sec:6}

On the numerical side, in monotone inclusion problems, 
the advantage of the resolvent
composition over compositions such as \eqref{e:c} or \eqref{e:1} is
that its resolvent is readily available through
Proposition~\ref{p:0}. Hence, processing it efficiently in
an algorithm does not require advanced splitting techniques.
In particular, in minimization problems, one deals with monotone
operators which are subdifferentials and handling a proximal
composition $\proxc{L}{g}$ is more straightforward than the
compositions $g\circ L$ or $L^*\pushfwd g$ thanks to
Example~\ref{ex:8}\ref{ex:8iii}.
On the modeling side, while these compositions are not 
interchangeable in general, replacing the standard composition 
\eqref{e:c} by a resolvent composition, may also be of interest.
For instance, in the special case of the basic proximal average 
\eqref{e:jj7}, replacing $g\circ L=\sum_{k=1}^p\omega_k g_k$
by $\proxc{L}{g}=(\sum_{k=1}^p\,\omega_k(g_k^*\infconv
\qq_{\HH}))^*-\qq_\HH$ in variational problems has been 
advocated in \cite{Liti20,Yuyl13}. More generally, the 
computational and modeling benefits of employing resolvent 
compositions in place of classical ones in concrete applications
is a natural topic of investigation, and it will be pursued 
elsewhere.  

The focus of this section is on the use of resolvent and proximal
compositions in the context of the following constrained inclusion
problem.

\begin{problem}
\label{prob:1}
Suppose that $L\in\BL(\HH,\GG)$ satisfies $0<\|L\|\leq 1$,
let $B\colon\GG\to 2^{\GG}$ be maximally monotone, and let
$V\neq\{0\}$ be a closed vector subspace of $\HH$. The task is to 
\begin{equation}
\label{e:p1}
\text{find}\;\:x\in V\;\:\text{such that}\;\:0\in B(Lx).
\end{equation}
\end{problem}

As will be illustrated in the examples below, \eqref{e:p1} models a
broad spectrum of problems in applied analysis. Of special interest
to us are situations in which, due to modeling errors,
$L(V)\cap\zer B=\emp$, which means that Problem~\ref{prob:1} has no
solution. As a surrogate to it with adequate approximate solutions
in such instances, we propose the following formulation. It is
based on the resolvent composition and will be seen to be solvable
by a simple implementation of the proximal point algorithm.

\begin{problem}
\label{prob:2}
Suppose that $L\in\BL(\HH,\GG)$ satisfies $0<\|L\|\leq 1$,
let $B\colon\GG\to 2^{\GG}$ be maximally monotone, let $V\neq\{0\}$
be a closed vector subspace of $\HH$, let $\gamma\in\RPP$, and
set $A=\proxcc{L}{(\gamma B)}$. The task is to 
\begin{equation}
\label{e:p2}
\text{find}\;\:x\in\HH\;\:\text{such that}\;\:
0\in\big(\proxc{\proj_V}{A}\big)x.
\end{equation}
\end{problem}

A justification of the fact that Problem~\ref{prob:2} is an
adequate relaxation of Problem~\ref{prob:1} is given in item
\ref{t:3vi} below. 

\begin{theorem}
\label{t:3}
Consider the settings of Problems~\ref{prob:1} and \ref{prob:2},
and let $S_1$ and $S_2$ be their respective sets of solutions.
Then the following hold:
\begin{enumerate}
\item
\label{t:3i}
$\proxc{\proj_V}{A}$ is maximally monotone.
\item
\label{t:3ii}
$J_{\proxc{\proj_V}{A}}=\proj_V\circ(\Id_\HH-L^*\circ L
+L^*\circ J_{\gamma B}\circ L)\circ\proj_V$.
\item
\label{t:3iii}
$S_1$ and $S_2$ are closed convex sets.
\item
\label{t:3iv}
$S_2=\Fix(\proj_V\circ(\Id_\HH-L^*\circ L
+L^*\circ J_{\gamma B}\circ L))$.
\item
\label{t:3vi}
Problem~\ref{prob:2} is an exact relaxation of
Problem~\ref{prob:1} in the sense that $S_1\neq\emp$ 
$\Rightarrow$ $S_2=S_1$.
\item
\label{t:3v}
$S_2=\zer(N_V+L^*\circ(\!\yosi{B}{\gamma})\circ L)$.
\end{enumerate}
\end{theorem}
\begin{proof}
\ref{t:3i}:
Theorem~\ref{t:1}\ref{t:1ic} asserts that $A$ is maximally
monotone. In view of Theorem~\ref{t:1}\ref{t:1i}, this makes
$\proxc{\proj_V}{A}$ maximally monotone.

\ref{t:3ii}: It follows from Proposition~\ref{p:0} and 
Proposition~\ref{p:1}\ref{p:1ic2} that 
\begin{align}
J_{\proxc{\proj_V}{A}}
&=\proj_V\circ J_{{\proxcc{L}{(\gamma B)}}}\circ\proj_V
\nonumber\\
&=\proj_V\circ\big(\Id_\HH-L^*\circ L
+L^*\circ J_{\gamma B}\circ L\big)\circ\proj_V.
\end{align}

\ref{t:3iii}: 
The maximal monotonicity of $B$ implies that $\zer B$ is closed and
convex \cite[Proposition~23.39]{Livre1}. Hence, since $L$ is
continuous and linear, $L^{-1}(\zer B)$ is closed and
convex, and so is therefore $S_1=V\cap L^{-1}(\zer B)$.
Likewise, it follows from \ref{t:3i} that
$S_2=\zer(\proxc{\proj_V}{A})$ is closed and convex.

\ref{t:3iv}:
It results from Lemma~\ref{l:1}\ref{l:1ii} and \ref{t:3ii} that
\begin{equation}
S_2=\zer\big(\proxc{\proj_V}{A}\big)
=\Fix J_{\proxc{\proj_V}{A}}=
\Fix\Big(\proj_V\circ\big(\Id_\HH-L^*\circ L
+L^*\circ J_{\gamma B}\circ L\big)\Big).
\end{equation}

\ref{t:3vi}:
Suppose that $\overline{x}\in S_1$ and $x\in S_2$. Then 
$\overline{x}=\proj_V\overline{x}$ and 
$0\in B(L\overline{x})$, i.e., by Lemma~\ref{l:1}\ref{l:1ii}, 
$L\overline{x}=J_{\gamma B}(L\overline{x})$ and therefore 
$\overline{x}=(\Id_\HH-L^*\circ L)\overline{x}+
L^*(L\overline{x})=(\Id_\HH-L^*\circ L)\overline{x}+
L^*(J_{\gamma B}(L\overline{x}))$. Altogether, bringing into play
\ref{t:3iv}, we get
\begin{equation}
\overline{x}=\proj_V\overline{x}=\proj_V\big((\Id_\HH-L^*\circ L)
\overline{x}+(L^*\circ J_{\gamma B}\circ L)\overline{x}\big)
\in S_2.
\end{equation}
It remains to show that $x\in S_1$, i.e., as \ref{t:3iv} yields 
$x\in V$, that $0\in B(Lx)$. Since
$L\overline{x}\in\zer B$, Lemma~\ref{l:1}\ref{l:1ii} entails that
$\yosi{B}{\gamma}(L\overline{x})=0$. Hence,
\begin{equation}
\label{e:s1}
(\forall v\in V)\quad
\Scal{v}{L^*\big(\!\yosi{B}{\gamma}(L\overline{x})\big)}=0.
\end{equation}
On the other hand, we derive from \ref{t:3iv} that
\begin{equation}
\label{e:t8}
x=\proj_V\Big(x-L^*\big((\Id_\HH-J_{\gamma B})(Lx)\big)\Big)
=\big(N_V+\Id_\HH\big)^{-1}
\Big(x-\gamma L^*\big(\yosi{B}{\gamma}(Lx)\big)\Big).
\end{equation}
Thus, $-L^*(\!\yosi{B}{\gamma}(Lx))\in N_Vx=V^\bot$, i.e., 
\begin{equation}
\label{e:s2}
(\forall v\in V)\quad
\Scal{v}{L^*\big(\!\yosi{B}{\gamma}(Lx)\big)}=0.
\end{equation}
Since $x-\overline{x}\in V$, we deduce from \eqref{e:s1} and
\eqref{e:s2} that 
\begin{equation}
\Scal{x-\overline{x}}{L^*\big(\!\yosi{B}{\gamma}(Lx)-
\yosi{B}{\gamma}(L\overline{x})\big)}=0.
\end{equation}
Thus,
\begin{equation}
\scal{Lx-L\overline{x}}{\!\yosi{B}{\gamma}(Lx)-\!\yosi{B}{\gamma}
(L\overline{x})}=0
\end{equation}
and, since $\yosi{B}{\gamma}$ is $\gamma$-cocoercive
\cite[Corollary~23.11(iii)]{Livre1}, we obtain
\begin{equation}
\gamma\|\yosi{B}{\gamma}(Lx)\|^2
=\gamma\|\yosi{B}{\gamma}(Lx)-
\yosi{B}{\gamma}(L\overline{x})\|^2\leq
\scal{Lx-L\overline{x}}{\yosi{B}{\gamma}(Lx)-
\yosi{B}{\gamma}(L\overline{x})}=0.
\end{equation}
We conclude that $\yosi{B}{\gamma}(Lx)=0$ and hence that
$Lx\in\zer\yosi{B}{\gamma}=\Fix J_{\gamma B}=\zer{B}$.

\ref{t:3v}: Let $x\in\HH$. Then, arguing as in \eqref{e:t8},
\begin{eqnarray}
x\in S_2
&\Leftrightarrow&
x-L^*\big(Lx-J_{\gamma B}(Lx)\big)\in(N_V+\Id_\HH)x
\nonumber\\
&\Leftrightarrow&
0\in N_Vx+L^*\big((\Id_\GG-J_{\gamma B})(Lx)\big)
\nonumber\\
&\Leftrightarrow&
x\in\zer\big(N_V+L^*\circ(\!\yosi{B}{\gamma})\circ L\big),
\end{eqnarray}
which provides the desired identity.
\end{proof}

\begin{remark}[isometry]
\label{r:5}
Suppose that $L$ is an isometry in Theorem~\ref{t:3} (see
Remark~\ref{r:2}). In view of Proposition~\ref{p:1}\ref{p:1iic} and
Proposition~\ref{p:9}, the relaxed problem \eqref{e:p2} is then to
find a zero of
\begin{equation}
\label{e:r4a}
\proxc{\proj_V}{A}
=\proxc{\proj_V}{\big(\proxc{L}{(\gamma B)}\big)}
=\proxc{(L\circ\proj_V)}{{(\gamma B)}}, 
\end{equation}
and it follows from Theorem~\ref{t:3}\ref{t:3iv} that its 
set of solutions is 
$S_2=\Fix(\proj_V\circ L^*\circ J_{\gamma B}\circ L)$. 
\end{remark}

Next, we propose an algorithm for solving Problem~\ref{prob:2}
which is based on the most elementary method for solving monotone
inclusions, namely the proximal point algorithm \cite{Roc76a}.

\begin{proposition}
\label{p:22}
Suppose that Problem~\ref{prob:2} has a solution, 
let $(\lambda_n)_{n\in\NN}$ be a sequence in $\left]0,2\right[$
such that $\sum_{n\in\NN}\lambda_n(2-\lambda_n)=\pinf$, and let
$x_0\in V$. Iterate
\begin{equation}
\label{e:9}
\begin{array}{l}
\text{for}\;n=0,1,\ldots\\
\left\lfloor
\begin{array}{l}
y_n=Lx_n\\
q_n=J_{\gamma B}y_n-y_n\\
z_n=L^*q_n\\
x_{n+1}=x_n+\lambda_n\proj_Vz_n.
\end{array}
\right.
\end{array}
\end{equation}
Then $(x_n)_{n\in\NN}$ converges weakly to a solution to 
Problem~\ref{prob:2}.
\end{proposition}
\begin{proof}
Set $M=\proxc{\proj_V}{(\proxcc{L}{(\gamma B)})}$.
Since $(x_n)_{n\in\NN}$ lies in $V$, it follows from
Theorem~\ref{t:3}\ref{t:3i}--\ref{t:3ii} that $(x_n)_{n\in\NN}$ is
generated by the proximal point algorithm, to wit, 
\begin{equation}
\label{e:34}
(\forall n\in\NN)\quad x_{n+1}=x_n+\lambda_n(J_Mx_n-x_n).
\end{equation}
Therefore, we derive from \cite[Lemma~2.2(vi)]{Joca09} that
$(x_n)_{n\in\NN}$ converges weakly to a point in $\zer M$, i.e., 
a solution to \eqref{e:p2}. 
\end{proof}

\begin{remark}[weak convergence]
\label{r:4}
The weak convergence of $(x_n)_{n\in\NN}$ in Proposition~\ref{p:22}
cannot be improved to strong convergence in general. Indeed,
suppose that, in Problem~\ref{prob:2}, $\GG=\HH$, $L=\Id_\HH$, and
$B=N_C$, where $C$ is a nonempty closed convex subset of $\HH$.
Then, if we take the parameters $(\lambda_n)_{n\in\NN}$ to be 1,
the proximal point algorithm \eqref{e:9} reduces to the alternating
projection method 
$(\forall n\in\NN)$ $x_{n+1}=\proj_V(\proj_C x_n)$. In
\cite{Hund04}, a hyperplane $V$ and a cone $C$ are constructed for
which $(x_n)_{n\in\NN}$ fails to converge strongly. Note, however,
that using the strongly convergent modifications of \eqref{e:34}
discussed in \cite{Moor01,Solo00}, it is straightforward to obtain
strongly convergent methods to solve Problem~\ref{prob:2}.
Let us add that, as shown in \cite[Lemma~2.2(vi)]{Joca09}, 
the weak convergence result in Proposition~\ref{p:22} remains 
valid if $q_n$ is defined as $q_n=J_{\gamma B}y_n+c_n-y_n$ in
\eqref{e:9}, where $(c_n)_{n\in\NN}$ is a sequence modeling
approximate implementations of $J_{\gamma B}$ and satisfies
$\sum_{n\in\NN}\lambda_n\|c_n\|<\pinf$.
\end{remark}

Henceforth, we specialize Problems~\ref{prob:1} and \ref{prob:2} to
scenarios of interest.

\begin{example}[feasibility problem]
\label{ex:76}
Let $0\neq m\in\NN$ and let $(\mathsf{C}_i)_{1\leq i\leq m}$ be
nonempty closed convex subsets of a real Hilbert space
$\mathsf{H}$. Set $\HH=\bigoplus_{i=1}^m\mathsf{H}$, 
$V=\menge{(\mathsf{x},\ldots,\mathsf{x})\in\HH}
{\mathsf{x}\in\mathsf{H}}$, and
$C=\mathsf{C}_1\times\cdots\times \mathsf{C}_m$.
Since $V$ is isomorphic to
$\mathsf{H}$, Problem~\ref{prob:1} with $\GG=\HH$, $L=\Id_\HH$, and
$B=N_C=A$ amounts to finding a point in $V\cap C$, i.e., a point in
$\bigcap_{i=1}^m\mathsf{C}_i$, while Theorem~\ref{t:3}\ref{t:3iv}
asserts that the relaxation given in Problem~\ref{prob:2} amounts
to finding a fixed point of $\proj_V\circ\proj_C$, i.e., of
$(1/m)\sum_{i=1}^m\proj_{\mathsf{C}_i}$ or, equivalently, a
minimizer of $\sum_{i=1}^md_{\mathsf{C}_i}^2$. This product space
framework for relaxing inconsistent feasibility problems was
proposed in \cite[Section~II.2]{Guyp76} and re-examined in
\cite{Baus93,Sign94}.
\end{example}

\begin{example}[resolvent mixtures]
\label{ex:23}
Let $0\neq p\in\NN$, let $\gamma\in\RPP$, and let $V\neq\{0\}$ be 
a closed vector subspace of $\HH$. For every $k\in\{1,\ldots,p\}$,
let $\GG_k$ be a real Hilbert space, let $L_{k}\in\BL(\HH,\GG_k)$,
let $\omega_k\in\RPP$, and let $B_k\colon\GG_k\to 2^{\GG_k}$ be
maximally monotone. Suppose that
$0<\sum_{k=1}^p\omega_k\|L_k\|^2\leq 1$ and define $\GG$, $L$, and
$B$ as in Example~\ref{ex:5}. Then the objective of
Problem~\ref{prob:1} is to 
\begin{equation}
\label{e:p3}
\text{find}\;\:x\in V\;\:\text{such that}\;\:
(\forall k\in\{1,\ldots,p\})\quad 0\in B_k(L_kx).
\end{equation}
Now let $M$ be the resolvent mixture of the operators
$((\gamma B_k)^{-1})_{1\leq k\leq p}$ (see Example~\ref{ex:5}).
Then the relaxed Problem~\ref{prob:2} is to 
\begin{equation}
\label{e:p14}
\text{find}\;\:x\in\HH\;\:\text{such that}\;\:
0\in\big(\proxc{\proj_V}{M^{-1}}\big)x
\end{equation}
or, equivalently, upon invoking Theorem~\ref{t:3}\ref{t:3v}, to
\begin{equation}
\label{e:p15}
\text{find}\;\:x\in\HH\;\:\text{such that}\;\:
0\in N_Vx+\sum_{k=1}^p\omega_kL_k^*
\big(\yosi{B_k}{\gamma}(L_kx)\big).
\end{equation}
In addition, it follows from Proposition~\ref{p:22} that, given
$x_0\in V$ and a sequence $(\lambda_n)_{n\in\NN}$ in
$\left]0,2\right[$ such that
$\sum_{n\in\NN}\lambda_n(2-\lambda_n)=\pinf$, the sequence
$(x_n)_{n\in\NN}$ constructed by the algorithm 
\begin{equation}
\label{e:91}
\begin{array}{l}
\text{for}\;n=0,1,\ldots\\
\left\lfloor
\begin{array}{l}
\text{for}\;k=1,\ldots,p\\
\left\lfloor
\begin{array}{l}
y_{k,n}=L_kx_n\\
q_{k,n}=J_{\gamma B_k}y_{k,n}-y_{k,n}\\
\end{array}
\right.\\
z_n=\sum_{k=1}^p\omega_kL_k^*q_{k,n}\\
x_{n+1}=x_n+\lambda_n\proj_Vz_n
\end{array}
\right.
\end{array}
\end{equation}
converges weakly to a solution to the relaxed problem if one 
exists.
\end{example}

\begin{example}[common zero problem]
\label{ex:60}
Suppose that, in Example~\ref{ex:23}, we have 
$(\forall k\in\{1,\ldots,p\})$ $\GG_k=\HH$ and $L_k=\Id_\HH$. Then
\eqref{e:p3} consists of finding 
$x\in V\cap\bigcap_{k=1}^p\zer B_k$
and its relaxation \eqref{e:p14}/\eqref{e:p15} consists of finding
a zero of $N_V+\sum_{k=1}^p\omega_k\yosi{B_k}{\gamma}$. This
relaxation was proposed in \cite{Opti04} and it originates in
Legendre's method of least-squares \cite{Lege05} to relax 
inconsistent systems of linear equations (see
\cite[Example~4.3]{Siop13}).
\end{example}

\begin{example}[Wiener systems]
\label{ex:61}
In Example~\ref{ex:23}, suppose that, for every 
$k\in\{1,\ldots,p\}$, 
$B_k=(\Id_{\GG_k}-F_k+p_k)^{-1}-\Id_{\GG_k}$, where
$F_k\colon\GG_k\to\GG_k$ is firmly nonexpansive and $p_k\in\GG_k$.
Then we recover the Wiener system setting investigated in
\cite{Siim22}. Specifically, \eqref{e:p3} reduces to the nonlinear 
reconstruction problem \cite[Problem~1.1]{Siim22}
\begin{equation}
\label{e:44}
\text{find}\;\:x\in V\;\:\text{such that}\;\:
(\forall k\in\{1,\ldots,p\})\quad F_k(L_kx)=p_k
\end{equation}
and \eqref{e:p14} yields the relaxed problem 
\cite[Problem~1.3]{Siim22}
\begin{equation}
\label{e:45}
\text{find}\:\;x\in V\:\;\text{such that}\:\;
(\forall y\in V)\;\:
\sum_{k=1}^p\omega_k\scal{L_ky-L_kx}{F_k(L_kx)-p_k}=0.
\end{equation}
In addition, given $x_0\in V$ and a sequence 
$(\lambda_n)_{n\in\NN}$ in $\left]0,2\right[$ such that
$\sum_{n\in\NN}\lambda_n(2-\lambda_n)=\pinf$, the sequence
$(x_n)_{n\in\NN}$ constructed by the algorithm 
\begin{equation}
\label{e:94}
\begin{array}{l}
\text{for}\;n=0,1,\ldots\\
\left\lfloor
\begin{array}{l}
\text{for}\;k=1,\ldots,p\\
\left\lfloor
\begin{array}{l}
y_{k,n}=L_kx_n\\
q_{k,n}=p_k-F_ky_{k,n}\\
\end{array}
\right.\\
z_n=\sum_{k=1}^p\omega_kL_k^*q_{k,n}\\
x_{n+1}=x_n+\lambda_n\proj_Vz_n
\end{array}
\right.
\end{array}
\end{equation}
converges weakly to a solution to the relaxed problem if one 
exists (see \cite[Proposition~4.3]{Siim22} for existence
conditions).
\end{example}
\begin{proof}
For every $k\in\{1,\ldots,p\}$, it follows from \eqref{e:f4} that 
$\Id_{\GG_k}-F_k+p_k\colon\GG_k\to\GG_k$ is firmly nonexpansive and
therefore from Lemma~\ref{l:0} that $B_k$ is maximally 
monotone, with $J_{B_k}=\Id_{\GG_k}-F_k+p_k$ and
$\yosi{B_k}{1}=F_k-p_k$. In addition, we
observe that this choice of the operators $(B_k)_{1\leq k\leq p}$
makes \eqref{e:44} a realization of \eqref{e:p3}, and 
\eqref{e:94} a realization of \eqref{e:91}. At the same time,
\eqref{e:p14}/\eqref{e:p15} with $\gamma=1$ becomes
\begin{equation}
\label{e:p19}
\text{find}\;\:x\in\HH\;\:\text{such that}\;\:
0\in N_Vx+\sum_{k=1}^p\omega_kL_k^*\big(F_k(L_kx)-p_k\big),
\end{equation}
which is precisely \eqref{e:45}.
\end{proof}

\begin{example}[proximal composition]
\label{ex:51}
In Problem~\ref{prob:1}, suppose that $B=\partial g$, where
$g\in\Gamma_0(\GG)$. Then
\eqref{e:p1} becomes
\begin{equation}
\label{e:p31}
\text{find}\;\:x\in V\;\:\text{such that}\;\:Lx\in\Argmin g.
\end{equation}
Now set $f=\proxcc{L}{(\gamma g)}$. Then the relaxation 
\eqref{e:p2} becomes
\begin{equation}
\label{e:p32}
\minimize{x\in\HH}{\big(\proxc{\proj_V}{f}\big)(x)}
\end{equation}
or, equivalently,
\begin{equation}
\label{e:p33}
\minimize{x\in V}{\big(\!\yosi{g}{\gamma}\big)(Lx)}.
\end{equation}
In addition, given $x_0\in V$ and a sequence
$(\lambda_n)_{n\in\NN}$ in $\left]0,2\right[$ such that
$\sum_{n\in\NN}\lambda_n(2-\lambda_n)=\pinf$, the algorithm 
\begin{equation}
\label{e:p39}
\begin{array}{l}
\text{for}\;n=0,1,\ldots\\
\left\lfloor
\begin{array}{l}
y_n=Lx_n\\
q_n=\prox_{\gamma g}y_n-y_n\\
z_n=L^*q_n\\
x_{n+1}=x_n+\lambda_n\proj_Vz_n
\end{array}
\right.
\end{array}
\end{equation}
produces a sequence $(x_n)_{n\in\NN}$ that converges weakly to a
solution to the relaxed problem if one exists.
\end{example}
\begin{proof}
The fact that \eqref{e:p1} yields \eqref{e:p31} is a consequence of
Fermat's rule \eqref{e:pierre}. Next, we derive from
Example~\ref{ex:7}\ref{ex:7ii} that, in Problem~\ref{prob:2},
\begin{equation}
A=\proxcc{L}{(\gamma B)}
=\proxcc{L}{\partial(\gamma g)}=\partial(\proxcc{L}{(\gamma g)})
=\partial f.
\end{equation}
Thus, by Example~\ref{ex:8}\ref{ex:8ii},
\begin{equation}
\label{e:p38}
\proxc{\proj_V}{A}
=\proxc{\proj_V}{\partial f}
=\partial\big(\proxc{\proj_V}{f}\big).
\end{equation}
Therefore, by Fermat's rule \eqref{e:pierre}, the solution set of 
Problem~\ref{prob:2} is
\begin{equation}
\label{e:p37}
\zer\big(\proxc{\proj_V}{A}\big)
=\Argmin\big(\proxc{\proj_V}{f}\big).
\end{equation}
On the other hand, since $\dom\yosi{g}{\gamma}=\GG$, 
\cite[Example~23.3 and Theorem~16.47(i)]{Livre1} yield
\begin{equation}
\label{e:p36}
N_V+L^*\circ\big(\yosi{\;(\partial g)}{\gamma}\big)\circ L
=\partial\iota_V+
L^*\circ\big(\nabla\yosi{g}{\gamma}\big)\circ L
=\partial\big(\iota_V+(\!\yosi{g}{\gamma})\circ L\big).
\end{equation}
Thus, we deduce from Theorem~\ref{t:3}\ref{t:3v} and
\eqref{e:pierre} that
\begin{equation}
\zer\big(\proxc{\proj_V}{A}\big)=
\zer\big(N_V+L^*\circ\yosi{\;(\partial g)}{\gamma}\circ L\big)
=\Argmin\big(\iota_V+(\!\yosi{g}{\gamma})\circ L\big).
\end{equation}
In view of \eqref{e:p38}, this confirms the equivalence between
\eqref{e:p32} and \eqref{e:p33}. The last claim is an application
of Proposition~\ref{p:22} using \eqref{e:dprox}.
\end{proof}

\begin{example}[proximal mixture]
\label{ex:52}
In the context of Example~\ref{ex:51}, choose $\GG$, $L$, and $g$
as in Example~\ref{ex:40}. Then the initial problem \eqref{e:p31}
is to
\begin{equation}
\label{e:p51}
\text{find}\;\:x\in V\;\:\text{such that}\;\:
(\forall k\in\{1,\ldots,p\})\;\;L_kx\in\Argmin g_k.
\end{equation}
Now let $m$ be the proximal mixture of the functions
$((\gamma g_k)^*)_{1\leq k\leq p}$ (see Example~\ref{ex:40}). 
Then the relaxation of \eqref{e:p51} given by \eqref{e:p32} is to
\begin{equation}
\label{e:p52}
\minimize{x\in\HH}{\big(\proxc{\proj_V}{m^*}\big)(x)}
\end{equation}
or, equivalently, via \eqref{e:p33}, to
\begin{equation}
\label{e:p53}
\minimize{x\in V}{\sum_{k=1}^p\omega_k
\big(\!\yosi{g_k}{\gamma}\big)(L_kx)}.
\end{equation}
This problem can be solved via \eqref{e:91}, where 
$J_{\gamma B_k}$ is replaced by $\prox_{\gamma g_k}$.
\end{example}

\begin{remark}[proximal average]
\label{ex:d3}
In Example~\ref{ex:51}, suppose that $V=\HH$ and that $L$ is an 
isometry. Then it follows from Proposition~\ref{p:12}\ref{p:12vi}
that the relaxed problem \eqref{e:p32} consists of minimizing the
proximal composition $f=\proxc{L}{(\gamma g)}$. In particular,
if $f$ is the proximal average of the functions 
$(g_k)_{1\leq k\leq p}$ (see \eqref{e:jj7}), it follows from
Example~\ref{ex:52} that minimizing it is an exact relaxation of
the problem of finding a common minimizer of the functions
$(g_k)_{1\leq k\leq p}$. This provides a principled interpretation 
for methodologies adopted in \cite{Liti20,Yuyl13}.
\end{remark}

\begin{example}[split feasibility]
\label{ex:62}
Suppose that, in Example~\ref{ex:52}, for every
$k\in\{1,\ldots,p\}$, $g_k=\iota_{D_k}$, where $D_k$ is a nonempty
closed convex subset of $\GG_k$. Then \eqref{e:p51} is the
split feasibility problem \cite{Reic20} 
\begin{equation}
\label{e:p62}
\text{find}\;\:x\in V\;\:\text{such that}\;\:
(\forall k\in\{1,\ldots,p\})\quad L_kx\in D_k,
\end{equation}
while the relaxation \eqref{e:p52}/\eqref{e:p53} is to 
\begin{equation}
\label{e:p63}
\minimize{x\in V}{\sum_{k=1}^p\omega_k d_{D_k}^2(L_kx)}.
\end{equation}
This problem can be solved via \eqref{e:91}, where 
$J_{\gamma B_k}$ is replaced by $\proj_{D_k}$.
\end{example}

\end{document}